\documentclass{amsart}
\usepackage{amsfonts, bbm}

\usepackage{graphicx}
\usepackage{psfrag}

\newtheorem{theorem}{Theorem}[section]
\theoremstyle{plain}
\newtheorem{corollary}[theorem]{Corollary}
\newtheorem{lemma}[theorem]{Lemma}
\newtheorem{proposition}[theorem]{Proposition}
\theoremstyle{remark}
\newtheorem{remark}[theorem]{Remark}

\numberwithin{equation}{section}

\newcommand{\tr}{\operatorname{tr}}
\newcommand{\otr}{\operatorname{0-tr}}
\newcommand{\vol}{\operatorname{vol}}
\newcommand{\ovol}{\operatorname{0-vol}}

\newcommand{\re}{\operatorname{Re}}
\newcommand{\im}{\operatorname{Im}}
\newcommand{\res}{\operatorname{Res}}
\newcommand{\rank}{\operatorname{rank}}

\newcommand{\supp}{\operatorname{supp}}

\newcommand{\norm}[1]{\Vert #1 \Vert}
\newcommand{\bnorm}[1]{\left\Vert #1 \right\Vert}
\newcommand{\brak}[1]{\langle #1 \rangle}

\newcommand{\bbR}{\mathbb{R}}
\newcommand{\bbH}{\mathbb{H}}
\newcommand{\bbC}{\mathbb{C}}
\newcommand{\bbZ}{\mathbb{Z}}
\newcommand{\bbN}{\mathbb{N}}

\newcommand{\calR}{\mathcal{R}}
\newcommand{\calH}{\mathcal{H}}
\newcommand{\calD}{\mathcal{D}}

\newcommand{\cinf}{C^\infty}
\newcommand{\del}{\partial}

\newcommand{\barX}{{\bar X}}
\newcommand{\bX}{{\partial\barX}}

\newcommand{\Ds}{(\Delta_g-s(n-s))}
\newcommand{\vep}{\varepsilon}
\newcommand{\tS}{\tilde{S}}

\newcommand{\nh}{\tfrac{n}{2}}
\newcommand{\chr}{\mathbbm{1}}
\newcommand{\bQ}{\mathbf{Q}}
\newcommand{\Ai}{{\rm Ai}}
\newcommand{\Bbkg}{B_n^{(0)}}
\newcommand{\Bvol}{B_P^{(1)}}
\newcommand{\Brad}{B_P^{(2)}}

\begin{document}

\title[Sharp upper bounds]{Sharp upper bounds on resonances for perturbations of
hyperbolic space}
\author[Borthwick]{David Borthwick}
\address{Department of Mathematics and Computer Science, Emory
University, Atlanta, Georgia, 30322, U. S. A.}
\thanks{Supported in part by NSF\ grant DMS-0901937.}
\email{davidb@mathcs.emory.edu}
\date{\today}
\subjclass[2000]{Primary 58J50,35P25; Secondary 47A40}

\begin{abstract}
For certain compactly supported metric and/or potential 
perturbations of the Laplacian on $\mathbb{H}^{n+1}$,
we establish an upper
bound on the resonance counting function with an explicit constant that depends only on
the dimension, the radius of the unperturbed region in $\mathbb{H}^{n+1}$, and the 
volume of the metric perturbation.
This constant is shown to be sharp in the case of scattering by a spherical obstacle.
\end{abstract}

\maketitle
\tableofcontents

%\section{Introduction}\label{intro.sec}

%%%%%%%
\bigbreak
\section{Introduction}\label{intro.sec}

For conformally compact manifolds that are hyperbolic near infinity, we now have 
fairly good control over the growth of the resonance counting function.  Upper and
lower bounds have been obtained for various cases in
\cite{Borthwick:2008, BCHP, CV:2003, GZ:1995b, GZ:1995a, GZ:1997, PP:2001, Perry:2003}.
In this paper we develop techniques which can provide a sharp constant for the
upper bound, and apply these specifically to the case where the manifold is a compactly supported perturbation
of the hyperbolic space $\bbH^{n+1}$.   The techniques are inspired by Stefanov's recent 
proof of sharp upper bounds on the resonance counting function
for perturbations of the Euclidean Laplacian \cite{Stefanov:2006}.

Let $\Delta_0$ denote the positive Laplacian on $\bbH^{n+1}$.
We can write the Green's function associated to $\Delta_0$ explicitly: if 
$R_0(s) := (\Delta_0 + s(n-s))^{-1}$, then
\begin{equation}\label{R0.def}
R_0(s;z,z') =  \frac{2^{-2s-1} \pi^{-\nh} \Gamma(s)}{\Gamma(s-\nh+1)} \sigma^{-s} 
F(s,s-\tfrac{n-1}2; 2s-n+1; \sigma^{-1}),
\end{equation}
where $F$ is the Gauss hypergeometric function and $\sigma := \cosh^2 (\tfrac12 d(z,z'))$.  From this expression
we quickly deduce that $R_0(s)$ admits an analytic extension to $s\in \bbC$ if $n$ is even,
and a meromorphic extension with poles at $s = -k$ for $k =0,1,2,\dots$ if $n$ is odd.  In the latter
case the multiplicities of the poles are given by 
\begin{equation}\label{m0.k}
m_0(-k) = (2k+1) \frac{(k+1)\dots(k+n-1)}{n!}.
\end{equation}
Let $\calR_0$ denote the resonance set for $\bbH^{n+1}$ (empty for $n$ even),
with resonances repeated according to multiplicity.
The associated resonance counting function is defined by
$$
N_0(t) := \#\{ \zeta\in \calR_0:\>|\zeta - \nh| \le t\}.
$$
For $n$ odd, an asymptotic for $N_0(t)$ is easily deduced by integrating (\ref{m0.k}).
For later usage, we introduce the constant
$$
\Bbkg := \begin{cases}
\frac{2}{(n+1)!} & n \text{ odd},\\
0 & n \text{ even}.
\end{cases}
$$
The resonance counting function asymptotics for $\bbH^{n+1}$ are then summarized by
\begin{equation}\label{N0.asymp}
N_0(t) \sim \Bbkg t^{n+1},
\end{equation}
as $t \to \infty$.

The main result of this paper concerns
the resonance counting function $N_P(t)$ for $P$ a
compactly supported perturbation of $\Delta_0$.  To describe the class of perturbations
precisely, let 
$$
K_0 := \overline{B(0; r_0)} \subset \bbH^{n+1}
$$   
for some $r_0 >0$.
We assume that $(X,g)$ is a smooth Riemannian manifold, possible with boundary, such that
for some compact $K \subset X$, we have
$$
(X-K, g) \cong (X_0 - K_0, g_0).
$$
In other words, $(X,g)$ agrees with $\bbH^{n+1}$ near infinity.
Note that $X$ is allowed to have a more complicated topology than $\bbH^{n+1}$,
as illustrated in Figure~\ref{Xg}.
%%fig: Xg
\begin{figure} 
\psfrag{r}{$r_0$}
\psfrag{K}{$K$}
\psfrag{H}{$\bbH^{n+1}$}
\begin{center}  
\includegraphics{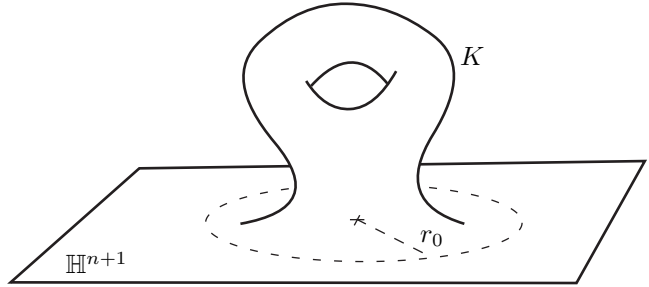} 
\end{center}
\caption{Compactly supported perturbation of $\bbH^{n+1}$, with $K$ replacing
the closed ball $K_0$.}\label{Xg}
\end{figure}

Let $\Delta_g$ denote the Laplacian on $(X,g)$, and $V \in \cinf_0(X)$ with
$\supp(V) \subset K$.
We then define the perturbed operator
$$
P := \Delta_g + V,
$$
where some self-adjoint boundary condition is imposed if $X$ has a boundary.
Since $R_0(s)$ functions as a good parametrix for $R_P(s): = (P-s(n-s))^{-1}$
near infinity, it is straightforward to prove meromorphic continuation of $R_P(s)$.
We can thus define the resonance set $\calR_P$, with resonances repeated
according to multiplicity, and the associated counting function
$$
N_P(t) := \#\{ \zeta\in \calR_P:\>|\zeta - \nh| \le t\}.
$$

The arguments of Cuevas-Vodev \cite{CV:2003} and Borthwick 
\cite{Borthwick:2008} are easily extended to show that 
$$
N_P(t) = O(t^{n+1}).
$$
Our goal in this paper is to refine this estimate by producing an explicit constant $B_P$ for this bound,
which is sharp in the sense that $N_P(t) \sim B_P t^{n+1}$ holds in at least some cases.

As in Stefanov's work \cite{Stefanov:2006}, such a result requires a slightly 
regularized version of the counting function.  
The basis of our estimate is the following relative counting formula:
\begin{equation}\label{rcf}
\int_0^a \frac{N_P(t) - N_0(t)}{t}\>dt = 2\int_0^a \frac{\sigma(t)}{t}\>dt + 
\frac{1}{2\pi} \int_{-\frac\pi2}^{\frac\pi2} \log |\tau(\nh+ ae^{i\theta})|\>d\theta + O(\log a),
\end{equation}
with $\tau(s)$ the relative scattering determinant for $P$ and $\sigma(t)$ the corresponding
relative scattering phase.  This formula holds for a general class of
background manifolds $(X_0, g_0)$ and for much more general perturbations; 
see Proposition~\ref{relcount}.

From the relative counting formula (\ref{rcf}), the role that the asymptotic 
(\ref{N0.asymp}) for $N_0(t)$ will play is clear.  
The contribution from the relative scattering phase $\sigma(t)$ is similarly easy to account for,
because it satisfies a Weyl-type asymptotic as $t \to \infty$,
\begin{equation}\label{sig.weyl}
\sigma(t) = \frac12 \Bvol\>t^{n+1} + O(t^n),
\end{equation}
where
$$
\Bvol :=  \frac{2 (4\pi)^{-\frac{n+1}2}}{\Gamma(\frac{n+3}2)} \bigl[\vol(K,g) - \vol(K_0,g_0)\bigr].
$$
It is for this result that we must require smoothness of $g$ and $V$.   
In various asymptotically hyperbolic settings, the scattering
phase asymptotic was established by Guilop\'e-Zworski \cite{GZ:1997}, Guillarmou \cite{Gui:2005b},
and Borthwick \cite{Borthwick:2008}.  By adapting of the arguments from \cite{Borthwick:2008},
we can extend the result to the class of perturbations considered here, for a general
class of background manifolds $(X_0, g_0$).

Once we have the scattering phase asymptotic, the final step in estimating the right-hand side of (\ref{rcf}) is to
study the integral of $\log |\tau(s)|$ over a half-circle.   It is here that we specialize to $\bbH^{n+1}$
as the background space.
With a combination of singular value techniques and asymptotic analysis of Legendre functions, we 
produce a bound
\begin{equation}\label{logt.est}
\frac{n+1}{2\pi} \int_{-\frac\pi2}^{\frac\pi2} \log |\tau(\nh+ ae^{i\theta})|\>d\theta
\le \Brad a^{n+1} + o(a^{n+1}),
\end{equation}
with
\begin{equation}\label{brad.def}
\Brad := \frac{n+1}{\pi \Gamma(n)} \int_{-\frac\pi2}^{\frac\pi2} \int_0^\infty  
\frac{[H(x e^{i\theta}, r_0)]_+}{x^{n+2}}\>dx\>d\theta,
\end{equation}
where $[\cdot]_+$ denotes the positive part and
\begin{equation}\label{Hdef}
\begin{split}
H(\alpha, r) & :=   \re \left[ 
2\alpha \log \Bigl( \alpha \cosh r +  \sqrt{1 + \alpha^2 \sinh^2 r} \Bigr) - \alpha \log (\alpha^2-1) \right] \\
&\qquad + \log \left| \frac{\cosh r - \sqrt{1 + \alpha^2 \sinh^2 r}}{\cosh r + \sqrt{1 + \alpha^2 \sinh^2 r}} \right|.
\end{split}
\end{equation}
The $r_0$-dependence of $\Brad$ is approximated by $\Brad \approx c_n e^{n r_0}$
for $r_0$ large, so that $\Brad$ is roughly proportional to $\vol(K_0,g_0)$.

The estimate (\ref{logt.est}) leads directly to our main result:
\begin{theorem}\label{main.thm}
For $P=\Delta_g +V$, a compactly supported perturbation of the Laplacian $\Delta_0$ on $\bbH^{n+1}$
as described above, we have
\begin{equation}\label{sharp.est}
(n+1) \int_0^a \frac{N_P(t)}{t}\>dt  
\le B_P a^{n+1} + o(a^{n+1}),
\end{equation}
where $B_P := \Bbkg + \Bvol + \Brad$.
\end{theorem}

To highlight the dependence of $B_P$ on $P$, we note that the constant $\Bbkg$ is dimensional,
$\Bvol$ depends on $r_0$ and on $(K,g)$ only through its volume, and $\Brad$ depends only on $r_0$.
None of these components depends on $V$.  

The factor $(n+1)$ is included in the formula (\ref{sharp.est}) so that an asymptotic
result of the same form as (\ref{sharp.est}) would be equivalent to 
$N_P(t) \sim B_P t^{n+1}$, with the same constant.
Note that the only missing ingredient needed to establish such an asymptotic result
is a lower bound of the same form as (\ref{logt.est}).  

To demonstrate the sharpness of Theorem~\ref{main.thm},
we consider explicitly the case of
scattering by a spherical obstacle in $\bbH^{n+1}$, for which $X = \bbH^{n+1}-
B(0; r_0)$, and $P = \Delta_0|_X$ with Dirichlet boundary conditions on $\del X$.
Figure~\ref{dresplot} shows a sample resonance set for a spherical obstacle in $\bbH^2$. 
%%fig: dresplot
\begin{figure} 
\psfrag{10}{$10$}
\psfrag{-10}{$-10$}
\begin{center}  
\includegraphics{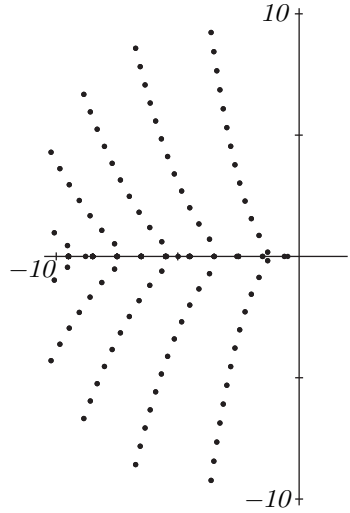} 
\end{center}
\caption{Resonances for the spherical obstacle of radius $r_0=1$ in $\bbH^2$.  All points off the real axis
have multiplicity two;  on the real axis the multiplicities are more complicated.}\label{dresplot}
\end{figure}
\begin{theorem}\label{dir.thm}
If $P$ is the Dirichlet Laplacian on $\bbH^{n+1} - B(0; r_0)$, then
$$
N_P(t) \sim B_P t^{n+1}.
$$
\end{theorem}
%%fig: Nobst
\begin{figure} 
\psfrag{10}{$10$}
\psfrag{400}{$400$}
\psfrag{800}{$800$}
\psfrag{rh}{$r_0 = \tfrac12$}
\psfrag{ro}{$r_0 = 1$}
\psfrag{rt}{$r_0 = 2$}
\begin{center}  
\includegraphics{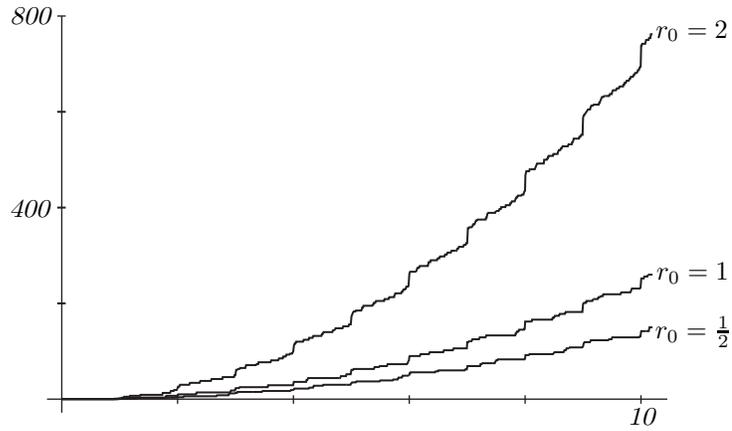} 
\end{center}
\caption{Resonance counting functions for spherical obstacles of radius $r_0$ in $\bbH^2$.}\label{Nobst}
\end{figure}
Figure~\ref{Nobst} shows that resonance counting functions $N_P(t)$ for spherical obstacles in $\bbH^2$ with several 
values of $r_0$.  These graphs are based on exact computation of the resonances.  
The approximate values of the asymptotic constants for these cases are
$$
B_P \approx \begin{cases}  1.45 & r_0 = \tfrac12, \\
2.61 & r_0 = 1, \\
7.50 & r_0 = 2.
 \end{cases}
$$
Already at $t=10$ we can see that the behavior of $N_P(t)$ is consistent with the predictions
of Theorem~\ref{dir.thm}.

The paper is organized as follows.  In \S\ref{bbox.sec} we develop basic
spectral results, such as meromorphic continuation of the resolvent,
in a very general ``black box'' perturbation setting.
In \S\ref{scatt.sec} we narrow the context somewhat, in order to
establish a nice factorization formula for the relative scattering determinant, from
which the relative counting formula (\ref{rcf}) follows.  Another application of the factorization is the 
Poisson summation formula for resonances, which leads to (\ref{sig.weyl}).  
The process of estimating the scattering determinant begins in \S\ref{pkernel.sec},
with a formula that expresses this determinant in terms of the Poisson kernel
on $\bbH^{n+1}$.  In \S\ref{tau.est.sec} we exploit this relation to prove 
(\ref{logt.est}) and complete the proof of Theorem~\ref{main.thm}.
A few explicit spherically symmetric examples are considered in~\S\ref{ex.sec},
which contains the proof of Theorem~\ref{dir.thm}.
Finally, the asymptotic analysis of Legendre functions
that is needed for \S\ref{tau.est.sec} and \S\ref{ex.sec} is developed in the Appendix. 

\vskip12pt\noindent
\textbf{Acknowledgment.}   I would like to thank to Plamen Stefanov for suggesting the extension
of his results to the hyperbolic setting.   That suggestion occurred during a workshop at 
the Banff International Research Station, and subsequently part of the work was carried out during a stay at 
the Centre International des Rencontres Math\'ematiques in Luminy.  I am grateful to both
of these institutions for their support, and also thank the Emory ICIS for additional travel support for
the Luminy meeting.

%%%%%%%%
\section{Black box perturbations}\label{bbox.sec}

In geometric scattering theory, the term ``black box'' refers to 
a general class of perturbations of the Euclidean Laplacian in $\bbR^n$ introduced by
Sj\"ostrand-Zworski \cite{SZ:1991}.  Although in standard usage this terminology is specific 
to the Euclidean setting, the same abstract formulation can be adapted to other settings.  
In this section, we will discuss black box perturbations in an asymptotically hyperbolic context.  
Our goal is to set up the definition of resonances by demonstrating
meromorphic continuation of the resolvent, and then to prove a global estimate of the
counting function.  It makes sense to do this in a general setting, since only minor changes are 
required to adapt previously published arguments. This section essentially amounts to
a review of known results.

An asymptotically hyperbolic metric on $(X_0,g_0)$ admits, by definition, a compactification $\barX_0$ 
with boundary defining function $\rho$ such that 
$(\barX_0, \rho^2 g_0)$ is a smooth, compact Riemannian manifold with boundary and
$|d\rho|_{\rho^2 g_0} = 1$ on $\bX_0$.  We will assume that $(X_0, g_0)$ is \textit{even} in the sense
introduced by Guillarmou \cite{Gui:2005a}.  This means
that the Taylor series of $\rho^2 g_0$ at $\rho=0$ contains only even powers of $\rho$.  
Under this assumption the resolvent 
$R_0(s) := (\Delta_{g_0}-s(n-s))^{-1}$ admits a meromorphic to $s \in \bbC$, with poles of finite rank
\cite{Gui:2005a, MM:1987}.

Appropriating the terminology from the Euclidean case, we define a class of
perturbations of $\Delta_{g_0}$ as follows.  Given a compact $K_0 \subset X_0$, we 
consider the Hilbert space
$$
\calH = \calH_0 \oplus L^2(X_0 - K_0, dg_0),
$$
where $\calH_0$ is some abstract Hilbert space filling in for $L^2(K_0, dg_0)$.
On $\calH$ we consider a self-adjoint operator $P$ 
with domain $\calD \subset \calH$, satisfying the following assumptions:
\begin{enumerate}
\item  $\calD|_{X_0 - K_0} \subset H^2(X_0 - K_0, dg_0)$.  If 
$u \in H^2(X_0 - K_0, dg_0)$ and $u$ vanishes near $K_0$, then $u \in \calD$.
\item  For $u \in \calD$,
$$
(Pu)|_{X_0 - K_0} = \Delta_{g_0} (u|_{X_0 - K_0}).
$$
\item  As a map $\calH\to \calH$, $\chr_{K_0} (P+ i)^{-1}$ is compact.
\end{enumerate}
Here the notations $\chr_{K_0}: u \mapsto u|_{K_0}$ and $\chr_{X_0 - K_0}: u \mapsto u|_{X_0 - K_0}$
denote the orthogonal projections $\calH \to \calH_0$ and $\calH \to L^2(X_0 - K_0, dg_0)$, respectively.

We will refer to an operator $P$ defined as above as a \textit{black box}
perturbation of $\Delta_{g_0}$.  Given that meromorphic continuation of the resolvent is
already known for $\Delta_{g_0}$, it is relatively easy to extend this result to $P$.

\begin{theorem}\label{pcontinue}
Let $(X_0, g_0)$ be an even asymptotically hyperbolic manifold and $P$
a black box perturbation of $\Delta_{g_0}$.
The resolvent $R_P(s) := (P - s(n-s))^{-1}$
admits for any $N$ a meromorphic continuation to $\re s > -N + \tfrac{n}2$
as an operator $\rho^{N} \calH \to \rho^{-N} \calH$, with poles of finite rank.
\end{theorem}
\begin{proof}
The resolvent $R_g(s)$ serves as a suitable parametrix for $R_P(s)$ near the boundary.
Let $\chi_0, \chi, \chi_1 \in \cinf_0(X)$ be cutoff functions equal to 1 on $K_0$,
such that $\chi=0$ on the support of $\chi_1$ and $\chi_0=1$ on the support of $\chi$.  
Let $K_1 := \supp \chi_1$, and define
$$
P_1 := P|_{K_1},
$$
as an operator on $\calH|_{K_1}$ with Dirichlet boundary conditions imposed on $\del K_1$,
so that $P_1$ is self-adjoint.  We can naturally regard $\chi_1(P_1-z)^{-1}\chi$ 
as an operator on $\calH$.

Then for $z_0$ such that $z_0 \notin \sigma(P)$ we set
\begin{equation}\label{Rparametrix}
M(s) = \chi_1 (P_1 - z_0)^{-1} \chi + (1-\chi_0) R_0(s) (1-\chi).
\end{equation}
Then
\begin{equation}\label{pmik}
(P - s(n-s)) M(s)  = I - L(s),
\end{equation}
where $L(s) = L_1(s_0) + L_2(s,z_0) + L_3(s)$ with
\[
\begin{split}
L_1(z_0)& := -[\Delta_{g}, \chi_1] (P_1 - z_0)^{-1} \chi, \\
L_2(s, z_0) & := (s(n-s) - z_0) \chi_1 (P_1 - z_0)^{-1} \chi, \\
L_3(s) & := [\Delta_{g}, \chi_0] R_0(s) (1-\chi).
\end{split}
\]
Our goal is to prove that $L(s)$ is compact and then apply the analytic Fredholm theorem.

Consider first the error term $L_1(s_0)$, which we can write as
$$
L_1(z_0) =  -[\Delta_{g}, \chi_1] \chr_{X-K} (P_1 - z_0)^{-1} \chi
$$
By definition, 
$\chr_{X_0-K_0} (P_1 - z_0)^{-1}$ maps $\calH$ to $\calD|_{X_0-K_0}$ and we have assumed
that the latter is contained in $H^2(X_0-K_0, dg_0)$.  
Since $[\Delta_{g_0}, \chi_0]$ is first order with smooth coefficients
whose compact support is contained in ${X_0-K_0}$, we see that $[\Delta_{g}, \chi_0]$ is compact
as a mapping $H^2(X_0-K_0, dg_0) \mapsto L^2(X_0-K_0, dg_0)$.  
Hence $L_1(z_0)$ is compact  $\calH \to \calH$.

The black box assumption that $\chr_{K_0} (P-i)^{-1}$ is compact implies that $\chr_{K_0} (P_1-i)^{-1}$
is compact on $\calH|_{K_1}$.  And the resolvent identity
\begin{equation}\label{res.id}
(P_1-z)^{-1} = (P_1-i)^{-1} \Bigl[ I + (z-i) (P_1-z)^{-1}  \Bigr]
\end{equation}
then shows that $L_2(s, z_0)$ is compact on $\calH$.
Finally, the error term $L_3(s)$ has a smooth kernel contained in
$\rho^\infty {\rho'}^s \cinf(X\times X)$.  This implies that for 
$N>0$, $L_3(s)$ is a compact operator
on $\rho^N \calH$ for $\re s > -N + \tfrac{n}2$.  

After adding the pieces together, these arguments show that 
$L(s)$ is compact on $\rho^N\calH$ for $\re s \ge -N + \nh$.
Using the self-adjointness of $P_1$ and the standard resolvent estimate, 
\begin{equation}\label{std.res}
\norm{ (P_1 - z)^{-1}} \le \frac{1}{\text{dist}( z, \sigma(P))}
\end{equation}
we can insure that $\norm{ (P_1 - z_0)^{-1}}$ is small by choosing $\im z_0$ large.   
Similarly, we can make $\norm{R_g(s)}$ small by choosing $s$ in the first quadrant sufficiently far
from the real axis and the line $\re s = \nh$.   Thus for some $s, z_0$ we have
$\norm{L(s)} < 1$, implying that $I-L(s)$ is invertible at this point.
The analytic Fredholm theorem then applies to define $(I - L(s))^{-1}$ 
meromorphically on $\rho^N \calH$ for $\re s > -N + \tfrac{n}2$.  The claimed result follows
from
$$
R_P(s) = M(s) (I-L(s))^{-1},
$$
because $M(s)$ maps $\rho^{N} \calH \to \rho^{-N} \calH$ for $\re s > -N + \tfrac{n}2$.
\end{proof}

\bigbreak
The fact that $R_P(s)$ admits meromorphic continuation as a bounded operator
on $\calH$ for $\re s > \nh$ (the $N=0$ case) implies, as an immediate corollary, that
\begin{equation}\label{sig.disc}
\sigma(P) \cap (-\infty, \tfrac{n^2}4) \text{ is discrete}.
\end{equation}

Theorem~\ref{pcontinue} allows us to define resonances associated to $P$ as
the poles of $R_P(s)$, with multiplicities given by
$$
m_P(\zeta)  := \rank \res_\zeta R_P(s).
$$
Then $\calR_P$ is defined to be the set of resonances of $P$, repeated according to the
multiplicities $m_P$.  The corresponding counting function is 
$$
N_P(t) :=  \#\{ \zeta\in \calR_P:\>|\zeta - \nh| \le t\}.
$$
The remaining goal of this section is to establish an order-of-growth estimate for $N_P(t)$.  
This requires first of all that $(X_0, g_0)$ be hyperbolic near infinity, in the sense that
sectional curvatures all equal $-1$ outside some compact set.  (No resonance bounds
are currently known in the asymptotically hyperbolic case without this extra condition.)
Such asymptotically hyperbolic manifolds are even in particular. 

We must also make some extra assumptions of $P$:
\renewcommand{\labelenumi}{(\roman{enumi})}
\begin{enumerate}
\item The operator $P$ must be bounded below, so
that the set (\ref{sig.disc}) is actually finite.  
\item The singular values of the resolvent of the cutoff operator $P_1$ introduced in the proof of Theorem~\ref{pcontinue}
satisfy a growth estimate,
\begin{equation}\label{p1.sing}
\mu_k((P_1-z)^{-1}) \le C\>|\im z|^{-\frac12} k^{-\frac{1}{n+1}},
\end{equation}
for some $C$ independent of $z$ and $k$. 
\end{enumerate}

The natural way to satisfy the growth estimate (\ref{p1.sing})
is to assume that $\calH = L^2(X, dg)$ for some Riemannian manifold $(X,g)$, possibly with boundary,
and that $P$ is an elliptic  self-adjoint pseudodifferential operator of order $2$.   
Then to establish (\ref{p1.sing}) we can start by using the resolvent estimate 
(\ref{std.res}) to estimate
$$
\mu_k((P_1-z)^{-1}) \le |\im z|^{-\frac12} \>\mu_k(|P_1-z|^{-\frac12}).
$$
Let $\Delta_{K_1}$ denote the Dirichlet Laplacian on $K_1$.
Since $|P_1-z|^{-\frac12}$ has order $-1$, the operator $(\Delta_{K_1}+1)^{\frac12} |P_1-z|^{-\frac12}$
is zeroth order and thus bounded on $L^2(K_1, dg)$.  
Then (\ref{p1.sing}) follows from 
\[
\begin{split}
\mu_k(|P_1-z|^{-\frac12}) & \le \mu_k((\Delta_{K_1}+1)^{-\frac12})
\>\bnorm{(\Delta_{K_1}+1)^{\frac12} |P_1-z|^{-\frac12} }\\
& \le C k^{-\frac{1}{n+1}}.
\end{split}
\]
(The fact that $C$ can be chosen independently of $z$ follows from the resolvent estimate (\ref{std.res}).)

\begin{theorem}\label{upper.bound}
Let $(X_0,g_0)$ be a conformally compact manifold, hyperbolic near infinity, and
$P$ a black box perturbation of $\Delta_{g_0}$ that satisfies the extra assumptions (i) and (ii).   Then 
$$
N_P(t) = O(t^{n+1}).
$$
\end{theorem}
\begin{proof}
This is a fairly minor generalization of the upper bound proved by Cuevas-Vodev \cite{CV:2003}
and Borthwick \cite{Borthwick:2008}.  This is because
for those arguments the interior metric enters only in the interior parametrix term, 
i.e., the first term on the right in (\ref{Rparametrix}).  The difficult part of the upper bound
analysis involves the terms supported near infinity, and this part of 
the argument applies immediately to $P$ by the assumption that 
$P|_{X_0-K_0} = \Delta_{g_0}$.

To apply the argument from Cuevas-Vodev, we need to check some estimates on the interior
error terms $L_1(z_0)$ and $L_2(s,z_0)$.  For the former, the fact that $\calD|_{X_0-K_0} \subset H^2(X_0-K_0, dg_0)$
implies that $L_1(z_0)$ is bounded as a map $\calH \to H^1(K_1-K_0, dg)$.  If $\Delta_{K_1-K_0}$ denotes
the Dirichlet Laplacian on $K_1-K_0$, then we can estimate
\begin{equation}\label{muk.L1}
\begin{split}
\mu_k(L_1(z_0)) & \le \mu_k\left((\Delta_{K_1-K_0}+1)^{-\frac12}\right)\> \bnorm{(\Delta_{K_1-K_0}+1)^{\frac12} L_1(s)} \\ 
& \le C k^{-\frac{1}{n+1}},
\end{split}
\end{equation}
where we can use (\ref{res.id}) and (\ref{std.res}) to see that we may take $C$ to be independent
of $z_0$.  For the $L_2(s, z_0)$ term, we first of all note that (\ref{std.res}) implies
$$
\norm{L_2(s,z_0)} \le C \frac{|s(n-s)-z_0|}{|\im z_0|}.
$$
By the assumption (\ref{p1.sing}) we can immediately estimate
\begin{equation}\label{muk.L2}
\mu_k(L_2(s,z_0)) \le C\frac{|s(n-s)-z_0|}{|\im z_0|^{\frac12}} \>k^{-\frac{1}{n+1}}.
\end{equation}
For the argument in \cite{CV:2003} one needs to
set $z_0 = \gamma N( n - \gamma N)$ for each $N$ such that $|s| \le N$, 
so the precise dependence of these estimates on $s$ and $z_0$ is significant.
The estimates (\ref{muk.L1}) and (\ref{muk.L2}) 
correspond precisely to the interior estimates \cite[eq.'s~(2.23--4)]{CV:2003}.
The proof of \cite[Prop.~1.2]{CV:2003} then gives a bound
$$
\#\Bigl\{\zeta\in \calR_P:\> |\zeta| \le r,\> \arg (\zeta - \tfrac{n}2) \in [-\pi+\varepsilon, \pi - \varepsilon]\Bigr\}
\le C_\varepsilon r^{n+1}
$$

To fill in the missing sector containing the negative real axis, 
we apply the argument from Borthwick \cite{Borthwick:2008}.  Here the interior parametrix enters
only in the proof of \cite[Lemma~5.2]{Borthwick:2008}.  
The required bound is that for some constant $a\ge n$, $\norm{R_{P}(s)} = O(1)$ for 
$\re s \ge a$.   Since $P$ is self-adjoint and bounded below
by assumption, this follows from the standard resolvent estimate.
The proof of \cite[Prop.~5.1]{Borthwick:2008} then shows that
$$
\#\Bigl\{\zeta\in \calR_P:\> |\zeta| \le r,\> \arg (\zeta+ a - n) \in [\tfrac{\pi}2+\varepsilon, 
\tfrac{3\pi}2 - \varepsilon]\Bigr\} \le C_\varepsilon r^{n+1}.
$$

The combination of estimates in the two regions gives the global result.
\end{proof}

\begin{remark}
Colin Guillarmou has noted a mistake in the original argument from \cite{GZ:1995b}, which
propagated through the arguments in \cite{CV:2003} and 
\cite{Borthwick:2008}.  The faulty claim is that one can choose a family of cutoffs 
$\{\chi^i\}$ such that $\sum \chi^i = 1$
in some neighborhood of $\bX$ and also so that, in local coordinates isometric to 
the unit half-disk in $\bbH^{n+1}$, $\chi^i$ factors as $\varphi(x) \psi(y)$ in the coordinates
$(x,y) \in \bbR^n \times \bbR_+$.  It is not possible to satisfy these assumptions simultaneously.

Fortunately, this problem is relatively easy to fix.  There are two sets of cutoffs used in these proofs.
(All three proofs use the same construction.)  The inner cutoffs $\{\chi^i\}$ must form
a partition of unity near the boundary, but are not actually required to factor in local coordinates.
The essential requirement for the inner cutoffs is that their
derivatives satisfy quasi-analytic estimates 
(see \cite[eq.~(2.6)]{Borthwick:2008} for example), and this is easily obtained without
reference to a factorization.  The local factorization assumption is crucial only
for outer cutoffs $\{\chi_1^i\}$ (with $\chi_1^i=1$ 
on the support of $\chi^i$).
We may keep this assumption in place because the outer
cutoffs do not form a partition of unity.   
\end{remark}

%%%%%%%
\bigbreak
\section{Relative scattering theory}\label{scatt.sec}

For this section we continue to assume, as in Theorem~\ref{upper.bound}, a 
conformally compact background manifold $(X_0, g_0)$ that is hyperbolic
near infinity.  The restriction $h_0 = \rho^2 g_0|_{\rho=0}$ defines a Riemannian metric on $\bX_0$,
whose conformal class is independent of $\rho$.  Thus $\bX_0$ is commonly referred to as the
``conformal infinity'' of $(X_0, g_0)$.  A black box perturbation $P$ shares the
same conformal infinity, since $P$ agrees with $\Delta_{g_0}$ outside a compact set.

The scattering matrices $S_P(s)$ and $S_0(s)$, associated to $P$ and $\Delta_{g_0}$,
respectively, are pseudodifferential operators on $\bX_0$ defined as in \cite{JS:2000, GZ:2003}.
Away from the diagonal, we can realize the kernel of the scattering matrix as a boundary
limit of the resolvent:
\begin{equation}\label{sc.reslim}
S_*(s;x,x') = \lim_{\rho, \rho' \to 0}  (\rho\rho')^{-s} R_*(s;z,z')\qquad\text{for }x \ne x'.
\end{equation}
where $*$ = $P$ or $0$.  (This relationship can be extended to the diagonal if one is sufficiently 
careful - see \cite{JS:2000}.)
This connection allows us to see that $S_P(s)$ and $S_0(s)$ differ by a smoothing operator, as follows.
By applying $R_P(s)$ to (\ref{pmik}) from the left, we obtain the identity
$$
R_P(s) = M(s) + R_P(s)L(s).
$$
Then taking boundary limits as in (\ref{sc.reslim}) gives the kernel of $S_P(s)$ on the left, while on the right
we obtain the kernel of $S_0(s)$ as the limit of $M(s)$, plus a smooth contribution from the $L(s)$ term.
This implies that the relative scattering matrix $S_P(s) S_0(s)^{-1}$ is determinant class,
and we define the relative scattering determinant
\begin{equation}\label{tau.def}
\tau(s) := \det S_P(s) S_0(s)^{-1}.
\end{equation}

Let $H_*(s)$ denote the Hadamard product over the resonance set $\calR_*$:
\begin{equation}\label{pz.def}
H_*(s) := \prod_{\zeta\in \calR_*} E\Bigl(\frac{s}{\zeta}, n+1\Bigr),
\end{equation}
where
$$
E(z,p) := (1-z) \exp \Bigl(z + \frac{z^2}z + \dots + \frac{z^p}p \Bigr).
$$

\begin{proposition}\label{detsrel.factor}
Assume that $(X_0,g_0)$ is conformally compact and hyperbolic near infinity, and
$P$ is a black box perturbation of $\Delta_{g_0}$ satisfying the extra assumptions
(i) and (ii) from \S\ref{bbox.sec}.
The relative scattering determinant admits a factorization
\begin{equation}\label{detsrel.pp}
\tau(s) = e^{q(s)} \frac{H_P(n-s)}{H_P(s)} \frac{H_0(s)}{H_0(n-s)},
\end{equation}
where $q(s)$ is a polynomial of degree at most $n+1$.  
\end{proposition}
\begin{proof}
Since structure near infinity is unchanged from $(X_0,g_0)$, the arguments of Guillarmou~\cite{Gui:2005}
relating the resolvent and scattering pole multiplicities apply to $P$.  Thus the proof of
\cite[Prop.~7.2]{Borthwick:2008} shows that (\ref{detsrel.pp}) holds with 
with $q(s)$ a polynomial of unknown degree. 

To control the degree, we use the fact that $P$ is bounded from below to obtain
$$
\norm{R_P(s)} = O(1), \text{ for }\re s \ge a,
$$
for some $a \ge n$.
Then the proof of  \cite[Lemma~5.2]{Borthwick:2008} gives
that
$$
|\vartheta_P(s)| < e^{C_{\eta}\brak{s}^{n+1}},
$$
for $\re s < a - n$ with dist$(s, -\bbN_0) >\eta$.  The same estimate applies to $\vartheta_0(s)$.
In the formula
$$
\vartheta_P(s) = e^{-q(s)} \frac{H_0(n-s)}{H_0(s)} \frac{H_P(s)}{H_P(n-s)}\>  \vartheta_0(s),
$$
the Hadamard products have order $n+1$.  Thus the $\vartheta_*(s)$ estimates 
imply that $|q(s)| \le C |s|^{n+1+\delta}$ in the half-plane $\re s < a-n$, for any $\delta>0$.
Since $q(s)$ is already known to be polynomial, the degree of $q(s)$ is at most $n+1$.  
\end{proof}

\bigbreak
One nice application of Proposition~\ref{detsrel.factor} is a Jensen-type formula connecting the resonance
counting functions to a contour integral involving the relative scattering determinant.
To state this we introduce the relative scattering phase of $P$, defined as
$$
\sigma(\xi) := \frac{i}{2\pi} \log \tau(\tfrac12+i\xi),
$$
with branches of the log chosen so that $\sigma(\xi)$ is continuous starting from $\sigma(0) = 0$.
By the properties of the relative scattering matrix, $\sigma(\xi)$ is real and $\sigma(-\xi) = - \sigma(\xi)$.

The following relative counting formula is the asymptotically hyperbolic analog of a 
formula developed by Froese \cite{Froese:1998} for Schr\"odinger operators in the Euclidean setting.

\begin{proposition}\label{relcount}
Assume that $P$ is a black box perturbation of $(X_0,g_0)$ as in Proposition~\ref{detsrel.factor}.
As $a\to \infty$,
$$
\int_0^a \frac{N_P(t) - N_0(t)}{t}\>dt = 2\int_0^a \frac{\sigma(t)}{t}\>dt + 
\frac{1}{2\pi} \int_{-\frac\pi2}^{\frac\pi2} \log |\tau(\nh+ ae^{i\theta})|\>d\theta + O(\log a).
$$
\end{proposition}

\begin{proof}
According to Proposition~\ref{detsrel.factor},
for $\re(s)>\nh$, $\tau(s)$ has zeros when $n-s \in \calR_P$ or $s \in \calR_0$ and
the latter case occurs only if $s(n-s)$ lies in the discrete spectrum of $\Delta_{g_0}$.
Likewise, poles of $\tau(s)$ for $\re s >\nh$ occur
when either $n-s \in \calR_0$ or $s \in \calR_P$, the latter only if $s(n-s)$ lies in the
discrete spectrum of $P$.  All of these are counted with multiplicity of course.

Let $\eta$ denote the contour $(\nh + t \exp(i[-\pi/2,\pi/2])) \cup [\nh+ it,\nh-it]$, as
shown in Figure~\ref{econtour}. 
Assuming $t$ is not the absolute value of a resonance in $\calR$ or $\calR_0$, we have
$$
\frac{1}{2\pi i} \oint_{\eta} \frac{\tau'}{\tau}(s)\>ds = N_P(t) - N_0(t) - 2d_P(t)+2 d_0(t),
$$
where $d_*(u)$ is the counting function for the (finite) set $\calR_* \cap (\nh,\infty)$ (the resonances
coming from the discrete spectrum).
%%fig: econtour
\begin{figure} 
\psfrag{t}{$t$}
\psfrag{e}{$\eta$}
\psfrag{nh}{$\nh$}
\begin{center}  
\includegraphics{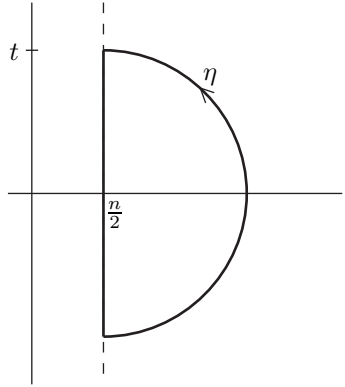} 
\end{center}
\caption{The contour $\eta$.}\label{econtour}
\end{figure}
Evaluating the contour integral yields
\[
\begin{split}
\frac{1}{2\pi i} \oint_{\eta} \frac{\tau'}{\tau}(s)\>ds
& = \im \frac{1}{2\pi} \oint_{\eta} \frac{\tau'}{\tau}(s)\>ds \\
& = \int_{-t}^t \sigma'(\xi)\>d\xi + \im \frac{1}{2\pi} \int_{-\frac\pi2}^{\frac\pi2} \frac{\tau'}{\tau}(\nh+te^{i\theta})\>ite^{i\theta}\>d\theta \\
& = 2\sigma(t) + \frac{1}{2\pi} \int_{-\frac\pi2}^{\frac\pi2} t\frac{\del}{\del t} \log |\tau(\nh+te^{i\theta})|\>d\theta.
\end{split} 
\]
Now if we divide by $t$ and integrate, we obtain the claimed formula with remainder
given by
$$
2\int_0^a \frac{d_P(t) - d_0(t)}{t}\>dt = O(\log a).
$$ 
\end{proof}

\bigbreak
Our second important application of Proposition~\ref{detsrel.factor} is to establish the Poisson formula,
which will lead to Weyl-type asymptotics for the relative scattering phase.
Define the meromorphic function $\Upsilon_*(s)$ by
$$
(2s-n) \otr [R_*(s) - R_*(1-s)],
$$
for $s\notin \bbZ/2$.  The connection between $\Upsilon_*(s)$ and the relative
scattering determinant established by Patterson-Perry \cite[Prop.~5.3 and Lemma~6.7]{PP:2001}
depends only on the structure of model neighborhoods near infinity, and so carries over
to our case without alteration.  This yields the following Birman-Krein type formula:
\begin{proposition}\label{birman.krein}
For $s\notin \bbZ/2$ we have the meromorphic identity,
$$
- \del_s \log \tau(s) = \Upsilon_P(s) - \Upsilon_0(s).
$$
\end{proposition}

By the functional calculus, $\Upsilon_*(\tfrac{n}2 + i\xi)$ is essentially the Fourier transform 
of the continuous part of the wave 0-trace (see \cite[Lemma~8.1]{Borthwick:2008} for
the precise statement).  By Propositions~\ref{detsrel.factor} and \ref{birman.krein} we can write
$$
\Upsilon_P(s) - \Upsilon_0(s) 
= \del_s \log \left[e^{q(s)} \frac{H_P(s)}{H_P(n-s)} \frac{H_0(n-s)}{H_0(s)} \right]
$$
Taking the Fourier transform just as in the proof of  \cite[Thm.~1.2]{Borthwick:2008}
then gives a relative Poisson formula:
\begin{theorem}\label{rel.poisson}
Assume that $P$ is a black box perturbation of $(X_0,g_0)$ as in Proposition~\ref{detsrel.factor}.
The difference of regularized wave traces satisfies
\[
\begin{split}
&\otr \left[ \cos \left(t \sqrt{P - \tfrac{n^2}4}\,\right) \right] 
- \otr \left[ \cos \left(t \sqrt{\Delta_{g_0} - \tfrac{n^2}4}\,\right) \right]  \\
&\qquad = \frac12 \sum_{\zeta\in \calR_P} e^{(\zeta-\frac{n}2)|t|} 
- \frac12 \sum_{\zeta\in \calR_0} e^{(\zeta-\frac{n}2)|t|},
\end{split}
\]
in the sense of distributions on $\bbR - \{0\}$.
\end{theorem}

The desired asymptotics of the scattering phase correspond to the big singularity of the wave trace
at $t=0$.  This singularity is very much analogous to that worked out by 
Duistermaat-Guillemin \cite{DG:1975} in the compact case.
The following result was proven for Riemann surfaces, possible with internal boundary, 
by Guillop\'e-Zworski \cite[Lemma~6.2]{GZ:1997} and for higher 
dimensional asymptotically hyperbolic manifolds (without boundary) 
by Joshi-S\'a Baretto \cite{JS:2001}.  

Let $(X_0, g_0)$ be a Riemannian manifold that is conformally compact and hyperbolic 
outside some compact set $K_0 \subset X_0$ (a more restrictive class than asymptotically
hyperbolic).   Then we consider another Riemannian manifold $(X,g)$, possibly with boundary, 
with compact $K \subset X$ such that $(X-K,g) \cong (X_0 - K_0, g_0)$.
Let $\Delta_g$ denote the Laplacian on $(X,g)$.
We may also include a potential $V \in \cinf_0(X)$, supported in $K$. 
Given this setup we define the operator
$$
P := \Delta_g + V,
$$  
acting on $L^2(X,dg)$ with some self-adjoint boundary condition imposed on 
the internal boundary $\del X$.  Clearly $P$ is a black box perturbation of
$\Delta_{g_0}$, and it satisfies assumptions (i) and (ii) of \S\ref{bbox.sec} by
the remark preceding Theorem~\ref{upper.bound}.

\begin{proposition}\label{wt.sing}
Assume that $P = \Delta_g + V$ as described above.  If $\psi \in \cinf_0(\bbR)$ has support 
in a sufficiently small neighborhood of $0$ and $\psi = 1$ in some smaller
neighborhood of $0$, then 
\begin{equation}\label{wave.asymp}
\int_{-\infty}^\infty e^{-it\xi} \psi(t) \otr \left[ \cos \left(t \sqrt{\smash[b]{P - n^2/4}}\,\right) \right]\>dt 
\sim \sum_{k=0}^\infty a_k |\xi|^{n-2k},
\end{equation}
where 
$$
a_0 =   \frac{2^{-n} \pi^{-\frac{n-1}2}}{\Gamma(\frac{n+1}{2})} \ovol(X,g).
$$ 
\end{proposition}
\begin{proof}
By finite speed of propagation we can use cutoffs to split the wave trace into 
internal and external pieces:
\[
\begin{split}
\otr \left[\cos \left(t \sqrt{\smash[b]{P - n^2/4}}\,\right)\right] & =
\tr \left[\cos \left(t \sqrt{\smash[b]{P_1 - n^2/4}}\,\right) \chi\right] \\
&\qquad + \otr \left[\cos \left(t \sqrt{\smash[b]{\Delta_{g_0} - n^2/4}}\,\right) (1-\chi)\right] .
\end{split}
\]
The small time behavior of the first term (which is an actual trace) is given by Ivrii's result for 
compact manifolds with boundary \cite{Ivrii:1980}.  For the exterior term we can apply
\cite{JS:2001}.
\end{proof}

Using Proposition~\ref{birman.krein} and the Fourier transform relationship between $\Upsilon_*(\xi)$
and the wave $0$-trace, we can extract from Proposition~\ref{wt.sing} the asymptotic behavior
of the relative scattering phase, defined as
$$
\sigma(\xi) := \frac{i}{2\pi} \log \tau(\tfrac12+i\xi),
$$
with branches of the log chosen so that $\sigma(\xi)$ is continuous starting from $\sigma(0) = 0$.
By the properties of the relative scattering matrix, $\sigma(\xi)$ is real and $\sigma(-\xi) = - \sigma(\xi)$.
\begin{corollary}\label{scphase.cor}
As $\xi \to +\infty$,
$$
\sigma(\xi) = \frac{(4\pi)^{-\frac{n+1}2}}{\Gamma(\frac{n+3}2)} \bigl[\vol(K,g) - \vol(K_0,g_0)\bigr]\>\xi^{n+1}
+ O(\xi^n).
$$
\end{corollary}

The argument to derive Corollary~\ref{scphase.cor} from Theorem~\ref{rel.poisson}
and Proposition~\ref{wt.sing} requires almost no change from that given for $n=1$ by Guillop\'e-Zworski 
\cite[Thm.~1.5]{GZ:1997}, so we omit the details.  Proposition~\ref{detsrel.factor} (and in particular
the bound on the order of $\tau(s)$) supplies the additional information needed to extend their
result to $n>1$.  The leading coefficient is initially given by a difference of
$0$-volumes, and we use $(X-K,g) \cong (X_0-K_0,g_0)$ to reduce this to a difference of
the volumes of $K$ and $K_0$.

%%%%%%%
\bigbreak
\section{Poisson kernel formulas}\label{pkernel.sec}

Since the asymptotics of $\sigma(t)$ are given by Corollary~\ref{scphase.cor}, application
of the formula from Proposition~\ref{relcount} requires only estimation of $|\tau(s)|$
in the half-plane $\re s > \nh$.  To facilitate this estimation, we need a more explicit
realization of $\tau(s)$ as a Fredholm determinant.  This realization will involve the Poisson
kernel for the background metric $(X_0,g_0)$.  For the moment we assume only
that $(X_0, g_0)$ is an even asymptotically hyperbolic metric.

The Poisson kernel can be derived from the kernel of the resolvent $R_0(s)$ by
the limit
$$
E_0(s;z,x') :=  \lim_{\rho' \to 0} {\rho'}^{-s} R_0(s;z,z'),
$$
for $z \in X_0$ and $x' \in \bX_0$.  This kernel defines the Poisson operator 
$$
E_0(s): L^2(\bX_0, dh) \to \rho^{-N} L^2(X_0,dg_0),
$$
for $\re s > -N + \nh$, where $h$ is the metric induced on $\bX_0$ by $\rho^2g_0$.
For $f \in \cinf(\bX)$ we can solve $(\Delta_{g_0} - s(n-s))u = 0$ by setting $u = E_0(s)f$.
Moreover,  for $\re s\ge\nh$ with $s(n-s)$ not in the discrete spectrum of $\Delta_{g_0}$,
$u$ has a two-part asymptotic expansion as $\rho \to 0$,
\begin{equation}\label{E0.asym}
(2s-n) E_0(s)f  \sim \rho^{n-s}f + \rho^{s} S_0(s)f,
\end{equation}
where $S_0(s)$ is the scattering matrix.  This expansion, for general choice of $f$, 
uniquely determines the scattering matrix via meromorphic continuation.

The same construction works for $S_P(s)$.  In particular, if we manage to
find a family of solutions of $(P - s(n-s))u = 0$ such that
$$
(2s-n) u \sim \rho^{n-s}f + \rho^{s} f',
$$
for $f\in \cinf(\bX_0)$ and $s$ in some suitable region, then $S_P(s)$ can be identified as the map $f \mapsto f'$.

\begin{lemma}\label{tauq}
Suppose that $P$ is a black box perturbation of $(X_0, g_0)$ with support in $K_0$.
Let $\chi_1,\chi_2 \in \cinf_0(X)$ be cutoff functions 
such that $K_0 \subset \{\chi_1=1\}$ and $\supp \chi_1\subset \{\chi_2=1\}$.
The relative scattering matrix can be written as the Fredholm determinant
$$
\tau(s) = \det (1+Q(s)),
$$
where
$$
Q(s) := (2s-n) E_0(s)^t  [\Delta_0, \chi_2] R_P(s) [\Delta_0,\chi_1]E_0(n-s).
$$
\end{lemma}

\begin{proof}
Since all of the operators in question are meromorphic families, we can restrict $s$ to some convenient
set like $\re s = \nh$, $s \ne \nh$ to avoid poles in the proof.

Given $f\in \cinf(\bX_0)$, consider the ansatz
\begin{equation}\label{up.def}
u = (1-\chi_1)E_0(s)f + u',
\end{equation}
as a solution of  $(P-s(n-s)) u = 0$.   Then $P(1-\chi_1) = \Delta_0(1-\chi_1)$ implies that
$$
-[\Delta_0,\chi_1]E_0(s)f + (P-s(n-s)) u' = 0.
$$
After applying $R_P(s)$ on the left, we see that $(P-s(n-s)) u = 0$ may be solved by setting
$$
u' = R_P(s) [\Delta_0,\chi_1]E_0(s)f.
$$

Using the assumption on supports of $\chi_1$ and $\chi_2$, we can derive
\[
\begin{split}
(\Delta_0 - s(n-s))(1-\chi_2) u' & = - [\Delta_0, \chi_2] u' + (1-\chi_2) (P-s(n-s))u' \\
& = - [\Delta_0, \chi_2]u'.
\end{split}
\]
This is compactly supported, so that $R_0(s)$ may be applied to give
\[
\begin{split}
(1-\chi_2) u' & = -R_0(s)[\Delta_0, \chi_2] u' \\
& = -R_0(s)[\Delta_0, \chi_2] R_P(s) [\Delta_0,\chi_1]E_0(s)f.
\end{split}
\]
From this we can deduce the asymptotic behavior of $u'$ as $\rho \to 0$,
$$
u' \sim -\rho^s E_0(s)^t[\Delta_0, \chi_2] R_P(s) [\Delta_0,\chi_1]E_0(s)f.
$$

Using the definition (\ref{up.def}) of $u$ and the known asymptotic (\ref{E0.asym}) for $E_0(s)f$, 
we thus derive the expansion
$$
(2s-n) u \sim \rho^{n-s}f + \rho^s S_0(s)f - \rho^s (2s-n) E_0(s)^t[\Delta_0, \chi_2] R_P(s) [\Delta_0,\chi_1]E_0(s)f .
$$
We can rewrite this as
\begin{equation}\label{u.ffp}
(2s-n) u \sim \rho^{n-s} f + (1+Q(s))S_0(s)f,
\end{equation}
using the identity
$$
E_0(s) = - E_0(n-s) S_0(s),
$$
which follows immediately from (\ref{E0.asym}).  From (\ref{u.ffp}) we read off that 
$$
S_P(s) = (1+Q(s))S_0(s),
$$
and the determinant follows.
\end{proof}

\bigbreak
In order to use Lemma~\ref{tauq} to estimate $\tau(s)$, we need explicit knowledge of the
background Poisson operator $E_0(s)$.  At this point we specialize to $(X_0, g_0) \cong \bbH^{n+1}$
and work out formulas for $E_0(s;z,x')$.  In the usual $\bbH^n$ coordinates, 
$z = (x,y) \in \bbR^n \times \bbR_+$, we can read off immediately from (\ref{R0.def}) that
$$
E_0(s;z,x') =  2^{-2s-1}\pi^{-\nh} \frac{\Gamma(s)}{\Gamma(s-\nh+1)} \biggl[\frac{y}{y^2+|x-x'|^2} \biggr]^s.
$$
However, our application requires that $E_0(s;\cdot,\cdot)$ be written in 
geodesic polar coordinates and then decomposed into spherical harmonics.  The easiest
way to do this is to rederive $E_0(s;\cdot,\cdot)$ from scratch.

In geodesic polar coordinates, $\bbH^{n+1} \cong \bbR_+ \times S^n$ and
the hyperbolic metric is given by
$$
g_0 = dr^2 + \sinh^2 r \>d\omega^2,
$$
where $d\omega^2$ denotes the standard sphere metric on $S^n$.  It is thus natural to 
adopt the boundary defining function 
\begin{equation}\label{rhor}
\rho = 2e^{-r},
\end{equation}
so that $h$, the metric induced on $\bX_0$ by $\rho^2 g_0$,
is also the standard sphere metric.

The Laplacian on $\bbH^{n+1}$ is
$$
\Delta_0 = -\frac{1}{\sinh^n r} \del_r (\sinh^n r\>\del_r) + \frac{1}{\sinh^2 r} \Delta_{S^n}.
$$
The eigenfunctions of $\Delta_{S^n}$ are spherical harmonics $Y_l^m$ with
$$
\Delta_{S^n} Y_l^m = l(l+n-1)Y_l^m.
$$
Here $l = 0, 1, 2, \dots$ and $m = 0, 1, \dots, h_n(l)$ with
\begin{equation}\label{ml.def}
h_n(l) := \frac{2l+n-1}{n-1} \binom{l+n-2}{n-2}.
\end{equation}

\begin{proposition}\label{E0.coeff}
For $\bbH^{n+1}$ the Poisson kernel in geodesic polar coordinates admits an expansion
\begin{equation}\label{E0.expand}
E_0(s;r,\omega,\omega') = \sum_{l=0}^\infty \sum_{m=1}^{h_n(l)} 
a_{l}(s;r) Y_l^m(\omega) \overline{Y_l^m(\omega')},
\end{equation}
with coefficients given by
$$
a_{l}(s;r) = 2^{\frac{n-1}2 - s} \pi^{1/2} \frac{\Gamma(l+s)}{\Gamma(s - \nh+1)}
(\sinh r)^{-\frac{n-1}2} P_{s-\frac{n+1}2}^{-l - \frac{n-1}2}(\cosh r),
$$
where $P_\nu^\mu(z)$ is the Legendre function.
\end{proposition}
\begin{proof}
If we expand the Poisson kernel with respect to the spherical harmonic
basis as in (\ref{E0.expand}), then the equation $\Ds E_0(s) = 0$ implies
the coefficient equations,
\begin{equation}\label{al.eq}
-\del_r^2 a_l - n \coth r\>\del_r a_l + \left(\frac{l(l+n-1)}{\sinh^2 r} - s(n-s)\right)a_l = 0.
\end{equation}
After a standard change of variables this becomes the Legendre equation.
Since the Poisson kernel is smooth in the interior, we 
select the Legendre solutions that are recessive for $r\to 0$, namely
$$
a_l(s; r) = A_l(s) (\sinh r)^{-\frac{n-1}2} P_{s-\frac{n+1}2}^{-l - \frac{n-1}2}(\cosh r),
$$
for some constants $A_l(s)$.

The constant $A_l(s)$ may be identified from the asymptotic expansion (\ref{E0.asym}) 
as $r \to \infty$.  For the coefficients this expansion implies that
$$
(2s-n) a_{l}(s;r) \sim \rho^{n-s} + [S_0(s)]_l \rho^s,
$$
with $[S_0(s)]_l(s)$ the matrix elements of the scattering matrix $S_0(s)$, which will be
diagonal in the spherical harmonic basis.  

Using (\ref{rhor}) and the well-known asymptotics of the Legendre $P$-function,
the leading terms in our ansatz as $r\to \infty$ are 
\[
\begin{split}
(\sinh r)^{-\frac{n-1}2} P_{s-\frac{n+1}2}^{-l - \frac{n-1}2}(\cosh r) & 
\sim 2^{s-\frac{n+1}2}\pi^{-1/2} \frac{\Gamma(s-\nh)}{\Gamma(l+s)} \rho^{n-s} \\
&\qquad + 2^{-s+\frac{n-1}2} \pi^{-1/2} \frac{\Gamma(\nh-s)}{\Gamma(l+n-s)} \rho^s,
\end{split}
\]
from which we deduce
$$
A_l(s) = 2^{\frac{n-1}2 - s} \pi^{1/2} \frac{\Gamma(l+s)}{\Gamma(s - \nh+1)}.
$$
\end{proof}

For future reference, note
that we can also read off from this construction the (well-known) matrix elements of $S_0(s)$,
\begin{equation}\label{S0.coeff}
[S_0(s)]_l = 2^{n-2s} \frac{\Gamma(\nh-s)}{\Gamma(s-\nh)} \frac{\Gamma(l+s)}{\Gamma(l+n-s)}.
\end{equation}

%%%%%%%
\bigbreak
\section{Scattering determinant estimates}\label{tau.est.sec}

In this section we will combine the formula for $\tau(s)$ from Lemma~\ref{tauq}
with the explicit Fourier coefficients of the Poisson kernel given in Proposition~\ref{E0.coeff}.
We can then use estimates of the Legendre $P$-function developed in the Appendix to
produce an estimate for the $|\tau(s)|$ term in the counting formula from 
Proposition~\ref{relcount}.

Throughout this section, the background metric is restricted to 
$(X_0, g_0) \cong \bbH^{n+1}$.  
We assume that $P$ is a black box perturbation of the hyperbolic Laplacian 
$\Delta_0$.  As in \S\ref{intro.sec}, the support of the perturbation is assumed to lie within
$$
K_0 := \{r \le r_0\}.
$$ 
The main result of this section is the following:
\begin{theorem}\label{inttau.prop}
Let $P$ be a black box perturbation of the hyperbolic Laplacian $\Delta_0$ on $\bbH^{n+1}$.
For $a \in \nh + \bbN$, we can estimate
$$
\frac{1}{2\pi} \int_{-\frac{\pi}2}^{\frac{\pi}2} \log |\tau(\nh+ ae^{i\theta})|\>d\theta
\le \Brad a^{n+1} + o(a^{n+1}),
$$
as $a \to \infty$,
where $\Brad$ was defined by (\ref{brad.def}).
\end{theorem}

Before proceeding with the proof, we note that the combination of Proposition~\ref{relcount},
Corollary~\ref{scphase.cor}, and Theorem~\ref{inttau.prop}, immediately yields the proof of
Theorem~\ref{main.thm}.  Note also that the restriction to $P = \Delta_g + V$ (from the more
general black box class) is needed only for Corollary~\ref{scphase.cor}.

\bigbreak
The proof of Theorem~\ref{inttau.prop} will be broken into several stages, starting with:
\begin{lemma}\label{tau.sum.lemma}
Assuming $P$ and $r_0$ are defined as above, fix some small $\vep > 0$ and $\eta>0$
and define $r_j := r_0 + j \eta$.  For $\re s \ge \nh$ with dist$(s(n-s), \sigma(P)) \ge \vep$,
the relative scattering determinant can be estimated by
\begin{equation}\label{tau.sum}
\log |\tau(s)| \le \sum_{l=0}^\infty h_n(l) \log \Bigl(1 + C \lambda_l(s) \Bigr),
\end{equation}
where
\begin{equation}\label{laml.def}
\lambda_l(s) := |2s-n| \>\left[ \int_{r_1}^{r_2} |a_l(n-s; r)|^2\>(\sinh r)^n\>dr \right]^{\frac12} 
\left[ \int_{r_2}^{r_3} |a_l(s; r)|^2\>(\sinh r)^n\>dr\right]^{\frac12} ,
\end{equation}
with $a_l(s;r)$ the coefficients from Proposition~\ref{E0.coeff}, and $C$ depends only on
$\vep$, $\eta$, and $r_0$.
\end{lemma}
\begin{proof}
Let $\chi_1$ and $\chi_2$ be smooth cutoffs as in Lemma~\ref{tauq},
such that $\chi_j = 1$ for $r \le r_j$ and $\chi_j = 0$ for $r \ge r_{j+1}$.  Then we can rewrite
the $Q(s)$ from Lemma~\ref{tauq} as
$$
Q(s) = (2s-n) E_0(s)^t \chr_{[r_2,r_3]} [\Delta_0, \chi_2] R_P(s) [\Delta_0,\chi_1]\chr_{[r_1,r_2]}E_0(n-s),
$$
where $\chr_{[r_i,r_{i+1}]}$ denotes the characteristic function $\chi_{[r_i,r_{i+1}]}(r)$, acting as
a multiplication operator.
By Lemma \ref{tauq} and the cyclicity of the trace we have
$$
\log |\tau(s)|  = 
\det \Bigl(1 + (2s-n) [\Delta_0, \chi_2] R_P(s) [\Delta_0,\chi_1]\chr_{[r_1,r_2]}E_0(n-s) E_0(s)^t \chr_{[r_2,r_3]}\Bigr). 
$$
For $\re s \ge \nh$, under the assumption dist$(s(1-s), \sigma(P)) \ge \vep$,
we can apply the spectral theorem and standard elliptic estimates to obtain
$$
\Bigl\Vert [\Delta_0, \chi_2] R_P(s) [\Delta_0,\chi_1] \Bigr\Vert \le C,
$$
where $C$ depends on $\vep$, $\eta$, and $r_0$.
Under these restrictions,
\begin{equation}\label{tau.sv}
\log |\tau(s)| \le \sum_{j=1}^\infty \log \Bigl(1 + C \mu_j(F(s)) \Bigr),
\end{equation}
where 
$$
F(s) := (2s-n) \>\chr_{[r_1,r_2]}E_0(n-s) E_0(s)^t \chr_{[r_2,r_3]}
$$

Using Proposition~\ref{E0.coeff}, the eigenfunctions of $F^*F(s)$ can then be written down explicitly. If
we define 
$$
u_{l,m}(r,\omega) :=  \chr_{[r_2,r_3]} \overline{a_l(s;r) Y_l^m(\omega)},
$$
then 
$$
F^*F(s) u_{l,m} = \lambda_l(s)^2 u_{l,m},
$$
where $\lambda_l(s)$ is given by (\ref{laml.def}).  To see that $\{\lambda_l(s)\}$,
counted with multiplicities, contains all
of the nonzero eigenvalues of $F^*F(s)$, suppose that $w \in L^2(\bbH^{n+1})$
and $\brak{u_{l,m}, w} = 0$ for all $l,m$.  Then by (\ref{E0.expand}) we have
$E_0(s)^t \chr_{[r_2,r_3]} w = 0$, which implies that $F^*F(s)w = 0$.  
Hence, after possible rearrangement, the sequences $\{\lambda_l(s)\}$ 
and $\{\mu_j(F(s))\}$ correspond.  
The claimed estimate follows from (\ref{tau.sv}). 
\end{proof}

\bigbreak
Lemma~\ref{laml.def} reduces our problem to the estimation of the $\lambda_l(s)$'s,
which we take up next.

\begin{lemma}\label{lambda.lemma}
Assume that $\re s > \nh$ and $|s - \nh| \in \bbN$, and
set $k = l + \tfrac{n-1}2$ and $k\alpha = s-\nh$.  Assuming that $r_3 \in (r_0,r_0+1)$
in the definition (\ref{laml.def}) of $\lambda_l(s)$, we have the bound
$$
\log \lambda_l(s) \le k H(\alpha, r_3) + C \log k,
$$
where $H(\alpha, r)$ was defined in (\ref{Hdef}),
with a constant $C$ that depends only on $r_0$.
\end{lemma}
\begin{proof}
From Proposition~\ref{E0.coeff} we obtain the explicit formula
\[
\begin{split}
\lambda_l(s) &= \Bigl| \sin \pi (s-\nh) \> \Gamma(l + s) \Gamma(l+n-s) \Bigr| 
\left[ \int_{r_1}^{r_2} \Bigl|P_{s-\frac{n+1}2}^{-l - \frac{n-1}2}(\cosh r)\Bigr|^2\>\sinh r\>dr \right]^{\frac12} \\
&\qquad \times\left[ \int_{r_2}^{r_3} \Bigl|P_{s-\frac{n+1}2}^{-l - \frac{n-1}2}(\cosh r)\Bigr|^2\>\sinh r\>dr\right]^{\frac12},
\end{split}
\]
where we have exploited the symmetry $P^{-k}_\nu(z) = P^{-k}_{-1-\nu}(z)$.

By conjugation, if necessary, we can assume that $\arg \alpha \in [0,\frac{\pi}2]$.
Applying the estimate from Corollary~\ref{PQ.bounds} then yields
\begin{equation}\label{lam.bound1}
\begin{split}
\lambda_l(s) & \le Ck^{\frac13} \left| \frac{\sin (\pi k\alpha) \> \Gamma(k(1+\alpha)+\tfrac12) \>
\Gamma(k(1-\alpha)+\tfrac12)}{\Gamma(k+1)^2} \right| \\
&\qquad 
\times\left[ \int_{r_1}^{r_2} e^{2k \re (\phi(\alpha, r)-p(\alpha))}\>\sinh r\>dr\right]^{\frac12} 
\left[ \int_{r_2}^{r_3} e^{2k \re (\phi(\alpha, r)-p(\alpha))}\>\sinh r\>dr\right]^{\frac12}.
\end{split}
\end{equation}
By (\ref{dzph}), $\re \phi(\alpha, r)$ is increasing as a 
function of $r$.
Hence
$$
\left[ \int_{r_j}^{r_{j+1}} e^{2k \re (\phi(\alpha, r)-p(\alpha))}\>\sinh r\>dr\right]^{\frac12} 
\le  e^{k \re (\phi(\alpha, r_3)-p(\alpha))} \cosh r_3,\quad j=1,2.
$$

The first factor on the right side of (\ref{lam.bound1}) can be estimated directly via
Stirling's formula for $\alpha \notin [1,\infty)$,
\begin{equation}\label{gamma.bd1}
\begin{split}
& \log \left| \frac{\sin \pi k\alpha \> \Gamma(k(1+\alpha)+\tfrac12) \>
\Gamma(k(1-\alpha)+\tfrac12)}{\Gamma(k+1)^2} \right|  \\
&\qquad = \pi k |\im \alpha| + k \re \Bigl[ (\alpha+1) \log (\alpha+1) + (1-\alpha) \log (1-\alpha) \Bigr]
+ O(\log k),
\end{split}
\end{equation}
as $k\to \infty$, uniformly for $\arg (\alpha-1) > \delta$.  
We can extend the same estimate to $\arg(\alpha-1) \le \delta$, using
$$
\sin \pi k\alpha\> \Gamma(k(1-\alpha)+\tfrac12) = \frac{-\pi \tan\pi k\alpha}{\Gamma(k(\alpha-1) + \tfrac12)},
$$
and our assumption that $|k\alpha| \in \bbN$, which implies
$$
|\tan \pi k\alpha| \le 1.
$$

After we note that
$$
H(\alpha, r) = \re \Bigl[ 2\phi(\alpha, r) - 2p(\alpha) + (\alpha+1) \log (\alpha+1) - (\alpha-1) \log (\alpha-1) \Bigr],
$$
we obtain from (\ref{lam.bound1}) and (\ref{gamma.bd1}) the estimate
$$
\log \lambda_l(s) \le k H(\alpha, r_3) + O(\log k) + 2 \log \cosh r_3.
$$
\end{proof}

\bigbreak
Now we can combine Lemmas~\ref{tau.sum.lemma} and \ref{lambda.lemma} to estimate $\tau(s)$.
The strategy here is similar to Stefanov's in  \cite[Thm.~5a]{Stefanov:2006}.
\begin{proposition}\label{logtau.prop}
For $a - \nh \in \bbN$ and  $|\theta| \le \tfrac{\pi}2$ we have
$$
\log |\tau(\nh+ae^{i\theta})| \le b(\theta, r_0) a^{n+1} + o(a^{n+1}),
$$
uniformly for $|\theta| \le \tfrac{\pi}2 - \vep a^{-2}$, with
\begin{equation}\label{bn.def}
b(\theta, r_0) :=  \frac{2}{\Gamma(n)} \int_{0}^\infty  \frac{[H(xe^{i\theta}, r_0)]_+}{x^{n+2}}\>dx,
\end{equation}
where $[\cdot]_+$ denotes the positive part.
\end{proposition}

\begin{proof}
Since
$$
s(n-s) = \frac{n^2}4 - a^2 e^{2i\theta},
$$
the assumption that $|\theta| \le \tfrac{\pi}2 - \vep a^{-2}$ implies that $(s(n-s)$
remains a distance $O(\vep)$ from $\sigma(P)$ for $a$ sufficiently large.
The hypothesis of Lemma~\ref{tau.sum.lemma} is thus satisfied, yielding 
the estimate (\ref{tau.sum}) with a $C$ that depends only
on $\vep$, $\eta$, and $r_0$.
To apply Lemma~\ref{lambda.lemma} to estimate the right-hand side of (\ref{tau.sum}), 
we need to distinguish the terms according to the sign of $H(\alpha, r_3)$.  
For large $a$ the sum is dominated by terms with $H(\alpha, r_3)>0$, which 
occurs for $\alpha$ outside a certain neighborhood of the origin, as shown in
Figure~\ref{hplot}.
%%fig: hplot
\begin{figure} 
\psfrag{rhp}{$H(\alpha, r) > 0$}
\psfrag{rh0}{$H(\alpha, r) = 0$}
\psfrag{1}{$1$}
\psfrag{At}{$A(\theta)$}
\begin{center}  
\includegraphics{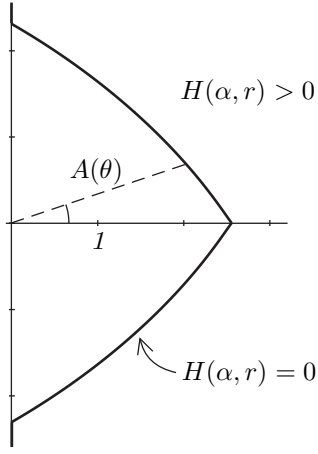} 
\end{center}
\caption{The positive region for $H(\alpha,r)$, shown for $r=1$.}\label{hplot}
\end{figure}

Let $x = A(\theta)$ be the implicit solution of the equation $H(xe^{i\theta}, r_3) = 0$,
so that $H(xe^{i\theta}, r_3) > 0$ precisely when $x > A(\theta)$.  
Given some $\delta>0$, we will subdivide the sum (\ref{tau.sum}) by breaking
at values where $|\alpha| = A(\theta)$ and $(1- \delta)A(\theta)$, leaving us with three parts.
The dominant part of the sum will be
$$
\Sigma_+ :=  \sum_{l:\> |\alpha| \ge A(\theta)} h_n(l) \log \Bigl(1 + C \lambda_l(s) \Bigr).
$$
(Recall that $\alpha = (s-\nh)/k$ where $k = l + (n-1)/2$.)
For $\alpha$ in this range, assuming $|\arg\alpha| \le \tfrac{\pi}2$,
we apply Lemma~\ref{lambda.lemma} to obtain
\begin{equation}\label{log1lam}
\log(1+C\lambda_l(s)) \le kH(\alpha, r_3) + C \log k
\end{equation}
Using this estimate together with the asymptotic
$$
h_n(l) = \frac{2l^{n-1}}{\Gamma(n)}(1 + O(l^{-1})),
$$
we have
\begin{equation}\label{sigp.est}
\Sigma_+ \le \sum_{k \le a/A(\theta)} \left(\frac{2k^{n-1}}{\Gamma(n)} +Ck^{n-2}\right) 
\biggl( kH\Bigl(\frac{ae^{i\theta}}{k}, r_3\Bigr) + C\log k \biggr).
\end{equation}
(The sum could be restricted to $k \ge \tfrac{n-1}2$, but this would not improve the bound.)
We can estimate $H(\alpha, r_3) = O(|\alpha|)$
with a constant  that depends only on $r_3$.   Thus
$$
 \sum_{k \le a/A(\theta)} k^{n-1} H\Bigl(\frac{ae^{i\theta}}{k}, r_3\Bigr) = O(a^n).
$$
With this estimate, the sums over lower order terms in (\ref{sigp.est}) are easily
controlled, and we obtain
\begin{equation}\label{sigplus.sum}
\Sigma_+ \le \frac{2}{\Gamma(n)} \sum_{k \le a/A(\theta)}  k^{n}
H\Bigl(\frac{ae^{i\theta}}{k}, r_3\Bigr)  + C a^n \log a,
\end{equation}
where $C$ depends only on $\vep$, $\eta$, and $r_0$.

Because $H(xe^{i\theta}, r)$ is an increasing function of $x$, 
the right-hand side of (\ref{sigplus.sum}) is easily estimated by the corresponding integral,
$$
\Sigma_+ \le  \frac{2}{\Gamma(n)} \int_{0}^{\frac{a}{A(\theta)}} k^n  H\Bigl(\frac{ae^{i\theta}}{k}, r_3\Bigr)\>dk
+ Ca^n \log a.
$$
Making the substitution $x = a/k$ gives
\[
\begin{split}
\Sigma_+ & \le \frac{2a^{n+1}}{\Gamma(n)} \int_{0}^\infty  \frac{[H(xe^{i\theta}, r_3)]_+}{x^{n+2}}\>dx
+ C a^n \log a \\
& =  b(\theta, r_3)a^{n+1} + C a^n \log a,
\end{split}
\]
with $C$ depending only on $\vep$, $\eta$, and $r_0$.

The middle term in (\ref{tau.sum}) will be
$$
\Sigma_0 :=  \sum_{l:\> (1-\delta)A(\theta) \le |\alpha| \le A(\theta)} 
h_n(l) \log \Bigl(1 + C \lambda_l(s) \Bigr).
$$
The number of terms in this sum is $O(a\delta)$, and we can control them using (\ref{log1lam}),
noting also that $H(\alpha, r_3) = O(\delta)$ for $|\alpha|$ in the given range.
Using an integral estimate as we did for $\Sigma_+$, we thus obtain
$$
\Sigma_0 \le C\delta a^{n+1} + C a^n \log a,
$$
where $C$ depends only on $\vep$, $\eta$, and $r_0$.

The final portion of the sum is
$$
\Sigma_- := \sum_{l:\> |\alpha| \le (1-\delta)A(\theta)} 
h_n(l) \log \Bigl(1 + C \lambda_l(s) \Bigr).
$$
We use the fact that $H(\alpha, r_3) \le -C\delta$ in this range to estimate
$$
\log (1+ C \lambda_l(s) ) \le C\lambda_l(s) \le Ce^{-ck}.
$$
This implies
$$
\Sigma_- \le C_\delta e^{-ca},
$$
for some $c>0$, where $C_\delta$ depends on $\delta$ as well as $\vep$ and the $r_j$'s.

Adding the three parts $\Sigma_+, \Sigma_0, \Sigma_-$ of (\ref{tau.sum}) together now yields
\begin{equation}\label{r3.delta}
\log |\tau(\nh+ae^{i\theta})| \le b(\theta, r_3) a^{n+1} + C(\vep,\eta,r_0) \bigl[ \delta a^{n+1}
+ a^n \log a\bigr] + C(\vep,\eta,r_0,\delta) e^{-ca},
\end{equation}
where we have made the dependence of the constants explicit.
Since $H(\alpha, r)$ is a strictly increasing function of $r$, we can absorb the $\delta a^{n+1}$
term into the first term by assuming that $\delta$ is small relative to $\eta$ and replacing $r_3$ with 
$r_4 = r_0 + 4\eta$.   With this change, we obtain from (\ref{r3.delta}) the estimate
\begin{equation}\label{r0.r4}
\frac{\log |\tau(\nh+ae^{i\theta})|}{a^{n+1}}  \le b(\theta, r_0 + 4\eta)  + 
C(\vep,\eta,r_0,\delta)a^{-1}\log a.
\end{equation}
The constant $C(\vep,\eta,r_0,\delta)$ may well blow up as $\eta \to 0$.   The best we can do
here is to observe that (\ref{r0.r4}) implies
$$
\limsup_{a \to \infty}\left[\frac{\log |\tau(\nh+ae^{i\theta})|}{a^{n+1}}  - b(\theta, r_0)\right]
\le  b(\theta, r_0 + 4\eta) - b(\theta, r_0),
$$
for any $\eta > 0$.
Since $b(\theta, r)$ is uniformly continuous on $[-\tfrac{\pi}2, \tfrac{\pi}2]\times [r_0, r_0+1]$, 
we can now let $\eta \to 0$ to obtain the claimed $o(a^{n+1})$ estimate.
\end{proof}

\bigbreak
\begin{proof}[Proof of Theorem~\ref{inttau.prop}] 
For any $\vep>0$, we can integrate the result from
Proposition~\ref{logtau.prop} over $|\theta| \le \tfrac{\pi}2 - \vep a^{-2}$, which gives,
$$
\int_{|\theta| \le \tfrac{\pi}2 - \vep a^{-2}}  \log |\tau(\nh+ ae^{i\theta})|\>d\theta 
\le \Brad a^{n+1} + o(a^{n+1}).
$$
The factorization given by Proposition~\ref{detsrel.factor}, 
together with the minimum modulus theorem (see e.g.~\cite[Thm.~8.71]{Titchmarsh}), 
implies that for any $\delta>0$, there exists a sequence $r_i \to \infty$ such that
\begin{equation}\label{tau.order.est}
|\tau(\nh + r_i e^{i\theta})| \le C_\delta \exp(r_i^{n+1+\delta}),
\end{equation}
uniformly in $\theta$.   
In sectors of the form $|\theta| \in [\tfrac{\pi}2 - \beta,  \tfrac{\pi}2]$, where $\tau(\nh+ae^{i\theta})$ is analytic,
we can apply a Phragm\'en-Lindel\"of argument, using (\ref{tau.order.est}),
$\log |\tau(s)| = 0$ for $\re s = \nh$, and the estimate from Proposition~\ref{logtau.prop} 
for $|\theta| = \tfrac{\pi}2 - \beta$, to conclude that
$$
|\tau(\nh + a e^{i\theta})| \le Ca^{n+1},
$$
uniformly for $|\theta| \in [\tfrac{\pi}2 - \beta,  \tfrac{\pi}2]$.  Thus, 
$$
\int_{\frac{\pi}2 - \vep a^{-2} \le |\theta| \le \frac{\pi}2}  \log |\tau(\nh+ ae^{i\theta})|\>d\theta
\le C \vep a^{n-1}.
$$
\end{proof}

%%%%%%%%
\bigbreak
\section{Examples}\label{ex.sec}

Suppose that $X= \bbH^{n+1}$ and we consider a black box perturbation $P = \Delta_g + V$, where
both the metric $g$ and potential $V$ are spherically symmetric.  The symmetry assumption 
guarantees that the perturbed Poisson kernel is ``diagonalized'' by spherical harmonics,
in the sense that
$$
E_P(s;r,\omega,\omega') = \sum_{l=0}^\infty \sum_{m=1}^{h_n(l)} 
a_{l}(s;r) Y_l^m(\omega) \overline{Y_l^m(\omega')},
$$
The coefficients $a_l(s;r)$ will satisfy (\ref{ml.def}) for $r > r_0$ and are thus expressible in terms
of Legendre functions.  Following the convention of Olver, we use the Legendre Q-function in
the form
\begin{equation}\label{bQ.def}
\bQ_\nu^\mu(z) := \frac{e^{-\mu\pi i}}{\Gamma(\nu+\mu+1)}\>  Q_\nu^\mu(z),
\end{equation}
where $Q_\nu^\mu(z)$ is the standard definition.  This makes $\bQ_\nu^\mu(z)$ an 
entire function of either $\mu$ or $\nu$, which is much more convenient for identifying resonances.
We can formulate the general solution of (\ref{ml.def}) for $r>r_0$ as
\begin{equation}\label{al.general}
a_l(s;r) = (\sinh r)^{-\frac{n-1}2} \biggl[ A_l(s)  \bQ_{\nu}^{k}(\cosh r) 
+ B_l(s) \bQ_{-\nu-1}^{k}(\cosh r) \biggr],
\end{equation}
where
\begin{equation}\label{k.nu}
k:= l + \frac{n-1}2, \qquad \nu := s - \frac{n+1}2.
\end{equation}
In particular examples, $A_l(s)$ and $B_l(s)$ will be determined by
matching $a_l$ and its first derivative to the corresponding solutions for $r < r_0$.
The scattering matrix elements can be read off from the asymptotics of these
solutions as $r\to \infty$, using
$$
a_l(s; r) \sim c_s \bigl(\rho^{n-s} +  \rho^s [S_P(s)]_l\bigr),
$$
in the same way that we found $[S_0(s)]_l$ in (\ref{S0.coeff}).  
Indeed, from the well-known asymptotic \cite[eq.~(12.09)]{Olver}
\begin{equation}\label{Q.asym}
\bQ_{\nu}^{k}(z) = \frac{\pi^\frac12}{\Gamma(\nu+\tfrac32)} \Bigl(\frac{z}2\Bigr)^{-\nu-1}
(1 + O(z^{-2})), \qquad \text{as }z\to\infty,
\end{equation}
we can see from (\ref{al.general}) that the scattering matrix elements are given by
\begin{equation}\label{SP.coeff}
[S_P(s)]_l = - 2^{n-2s} \frac{\Gamma(\nh-s)}{\Gamma(s-\nh)}
\frac{A_l(s)}{B_l(s)}.
\end{equation}

Consider the case where $P$ is the Laplacian for a spherical obstacle of radius $r_0$
in $\bbH^{n+1}$.  Imposing the Dirichlet boundary condition at $r=r_0$ gives 
coefficients
$$
A_s = \bQ_{-\nu-1}^{k}(\cosh r_0)  ,\qquad B_s = - \bQ_{\nu}^{k}(\cosh r_0).
$$
In this case, from (\ref{SP.coeff}) we see that
\begin{equation}\label{obst.S}
[S_P(s)]_l = 2^{n-2s} \frac{\Gamma(\nh-s)}{\Gamma(s-\nh)}
\frac{\bQ_{-\nu-1}^{k}(\cosh r_0)}{\bQ_{\nu}^{k}(\cosh r_0)}.
\end{equation}
With this observation we can give the:
\begin{proof}[Proof of Theorem~\ref{dir.thm}]
Our goal is to show that
\begin{equation}\label{logtau.asym}
\frac{1}{2\pi} \int_{-\frac\pi2}^{\frac\pi2} \log |\tau(\nh+ ae^{i\theta})|\>d\theta
\sim B_n(r_0) a^{n+1}.
\end{equation}
In conjunction with Proposition~\ref{relcount} and Corollary~\ref{scphase.cor}, (\ref{logtau.asym}) 
would imply that
$$
(n+1) \int_0^a \frac{N_P(t)}{t}\>dt \sim B_P a^{n+1},
$$ 
and this is equivalent to the stated asymptotic for $N_P(t)$.

Using (\ref{obst.S}) with (\ref{S0.coeff}) gives 
the relative scattering matrix elements,
$$
[S_P(s)S_0(s)^{-1}]_l =  \frac{\Gamma(l+n-s)}{\Gamma(l+s)}
\frac{\bQ_{-\nu-1}^{k}(\cosh r_0)}{\bQ_{\nu}^{k}(\cosh r_0)},
$$
with $k$ and $\nu$ defined as in (\ref{k.nu}).
With the connection formula \cite[eq.~(12.12)]{Olver},
$$
\frac{\bQ_{-\nu-1}^{k}(z)}{\Gamma(k+\nu+1)}  - \frac{\bQ_{\nu}^{k}(z)}{\Gamma(k-\nu)}
= \cos (\pi \nu) \>P_\nu^{-\mu}(z),
$$
we can rewrite the coefficient in the form 
\begin{equation}\label{dir.SS}
[S_P(s)S_0(s)^{-1}]_l =  1 - \cos (\pi \nu)\> \Gamma(k-\nu)
\frac{P_{\nu}^{-k}(\cosh r_0)}{\bQ_{\nu}^{k}(\cosh r_0)}.
\end{equation}

Now consider
$$
\log |\tau(s)| = \sum_{l=0}^\infty h_n(l) \log \Bigl|[S_P(s)S_0(s)^{-1}]_l\Bigr|.
$$
Defining $\alpha$ by $k\alpha = s - \nh$, we can use (\ref{dir.SS}) to write this as
\begin{equation}\label{lt.eta}
\log |\tau(s)| = \sum_{l=0}^\infty h_n(l) \log |1 - \eta_k(\alpha)|,
\end{equation}
where
$$
\eta_k(\alpha) := \sin (\pi k\alpha)\> \Gamma(k(1-\alpha)+\tfrac12)\>
\frac{P_{-\frac12+k\alpha}^{-k}(\cosh r_0)}{\bQ_{-\frac12+k\alpha}^{k}(\cosh r_0)}.
$$
Assuming that $|\arg \alpha| \le \tfrac{\pi}2 - \vep$, Corollary~\ref{PoverQ} gives
the estimate
$$
\log |\eta_k(\alpha)| \asymp  \left|\frac{\sin (\pi k\alpha)\> \Gamma(k(1-\alpha)+\tfrac12) \Gamma(k\alpha+1)}
{\Gamma(k+1)} \right|  e^{k \re[2\phi(\alpha, r_0) - p(\alpha) - q(\alpha)]},
$$
with constants depending only on $\vep$.
Applying Stirling's formula and avoiding the poles by assuming $|s - \nh| \in \bbN$ 
as in the proof of Lemma~\ref{lambda.lemma}, we have
$$
\frac{\sin (\pi k\alpha)\> \Gamma(k(1-\alpha)+\tfrac12) \Gamma(k\alpha+1)}
{\Gamma(k+1)} = k \re \Bigl[ \alpha \log \alpha - (\alpha-1) \log (\alpha-1) \Bigr] + O(\log \alpha)
$$
Since
$$
\re\Bigl[ 2\phi(\alpha, r_0) - p(\alpha) - q(\alpha) + \alpha \log \alpha - (\alpha-1) \log (\alpha-1)\Bigr] 
= H(\alpha, r_0),
$$
the full estimate is
\begin{equation}\label{etak.est}
\log |\eta_k(\alpha)| \asymp  kH(\alpha, r_0) + O(\log \alpha).
\end{equation}

As in the proof of Proposition~\ref{logtau.prop} we divide the sum (\ref{lt.eta}) into 
three pieces $\Sigma_+, \Sigma_0$, and $\Sigma_-$,
with breaks at $|\alpha| = (1\pm \delta)A(\theta)$ for some $\delta>0$.
The dominant piece is
$$
\Sigma_+ :=  \sum_{l:\> |\alpha| \ge (1+\delta)A(\theta)} h_n(l)
\log |1 - \eta_k(\alpha)|
$$
Using the lower bound from (\ref{etak.est}), but otherwise arguing as in
the proof of Proposition~\ref{logtau.prop}, we have
$$
\Sigma_+ \ge b(\theta, r_0)a^{n+1} + O(a^n \log a)
$$
The estimates on $\Sigma_0$ and $\Sigma_-$ are identical to those in
Proposition~\ref{logtau.prop}:
$$
\Sigma_0 \le c\delta a^{n+1} + O(a^n \log a),
$$
and 
$$
\Sigma_- = O(e^{-ca}).
$$
Hence we conclude that
$$
\log |\tau(\nh+ae^{i\theta})| \ge (b(\theta, r_0)-c\delta)a^{n+1} + O(a^n \log a),
$$
for $a \in \bbN$ and $|\theta| \le \tfrac{\pi}2 - \vep$, with constants that depend only
on $r_0$ and $\vep$.

Integrating, over $\theta$, and using Proposition~\ref{logtau.prop} to control the errors
from $|\theta| \in [\tfrac{\pi}2 - \vep, \tfrac{\pi}2]$, we obtain the estimate
$$
\frac{1}{2\pi} \int_{-\frac\pi2}^{\frac\pi2} \log |\tau(\nh+ ae^{i\theta})|\>d\theta
\ge (\Brad-\epsilon) a^{n+1} - C_{r_0,\epsilon} a^n \log a,
$$
valid for any $\epsilon >0$.  (This $\epsilon$ combines the terms proportional
to $\vep$ and $\delta$ from above.)
In combination with Theorem~\ref{inttau.prop}, this proves (\ref{logtau.asym}).
\end{proof}

\bigbreak
With some care, the explicit scattering matrix provided by (\ref{SP.coeff}) can
be used to compute resonances.  
Scattering poles and zeros are defined a renormalized scattering
matrix $\tS_P(s)$, in which the infinite rank poles and zeros coming from the gamma functions
are removed,
$$
\tS_P(s) := \frac{\Gamma(s-\nh)}{\Gamma(\nh-s)} S_P(s).
$$
The scattering multiplicity is then defined by 
$$
\nu_P(\zeta) := - \tr \bigl[\res_\zeta \tS_P'(s) \tS_P(s)^{-1}\bigr].
$$
For $\bbH^{n+1}$, the connection between scattering multiplicities and resonances is given by
\cite{GZ:1997, BP:2002, Gui:2005}
$$
\nu_P(\zeta) = m_P(\zeta) - m_P(n-\zeta) 
+ \begin{cases}0 & n\text{ odd}\\
\sum_{l\in \bbN} \Bigl( \mathbbm{1}_{n/2 - l}(\zeta) - \mathbbm{1}_{n/2+l}(\zeta) \Bigr) h_{n+1}(l)
& n \text{ even} \end{cases}
$$
For $\re \zeta < \nh$, the term $m_P(n-\zeta)$ plays a role only if $P$ has discrete spectrum.
In the examples that we will consider explicitly, $n=1$ and the discrete spectrum is empty.
For these cases, the resonances are precisely the poles of the $\tS_P(s)$.

Consider first the spherical obstacle of radius $r_0$ in $\bbH^2$, for which the scattering matrix
is given by (\ref{obst.S}).  From this expression we can read off the resonance set
$$
\calR_P = \bigcup_{k\in \bbZ} \bigl\{s: \bQ^{k}_{s-1}(\cosh r_0) = 0\bigr\}.
$$
Figures~\ref{dresplot} and \ref{Nobst} were thus obtained through numerical computation of zeroes
of the Legendre Q-function.

As a second example, we consider scattering in $\bbH^2$ by a radial step potential of the form
$$
V(r) = \begin{cases} c & r \le r_0, \\ 
0 & r > r_0. \end{cases}
$$
In this case,  with $P = \Delta_0 + V$, 
the coefficient solutions for $r \le r_0$ are Legendre $P$ functions $P^{-k}_{\omega(s)}(r_0)$,
with 
$$
\omega(s) := -\tfrac12 + \sqrt{(s-\tfrac12)^2+c}.
$$
The corresponding resonance set is
\begin{equation}\label{R.pot}
\calR_P = \bigcup_{k\in \bbZ} \Bigl\{s: \mathcal{W}\bigl[{\bQ^{k}_{s-1}}(z), 
P^{-k}_{\omega(s)}(z)\bigr]\Big|_{z = \cosh r_0} 
= 0\Bigr\}.
\end{equation}
where $\mathcal{W}$ is the Wronskian.   Resonance counting functions for $c=1$ and $c=5$, with $r_0 = 1$,
are shown in Figure~\ref{Npot}.  

We should note that Theorem~\ref{main.thm} does not apply to the step potential, because the lack 
of smoothness means that we cannot derive scattering phase asymptotics through Corollary~\ref{scphase.cor}. 
In view of the scattering phase asymptotics proved by Christiansen \cite{Christ:1998} 
in the black box Euclidean case, one might hope that the smoothness requirement in our case
could be loosened.  However, the technique of Robert used in \cite{Christ:1998} does not seem to be
applicable to the conformally compact hyperbolic case.
In any case, it is interesting to compare the putative upper bound suggested
by Theorem~\ref{main.thm} to the empirical results based on (\ref{R.pot}).  
For both of the cases shown in Figure~\ref{Npot}, the constant from the theorem would be
$B_P \approx 3.15$.   The numerical results thus suggest that $N_P(t)$ satisfies an asymptotic
with a constant significantly smaller than upper bound that Theorem~\ref{main.thm} would predict.
%%fig: Npot
\begin{figure} 
\psfrag{10}{$10$}
\psfrag{100}{$100$}
\psfrag{200}{$200$}
\psfrag{c1}{$c = 1$}
\psfrag{c5}{$c = 5$}
\begin{center}  
\includegraphics{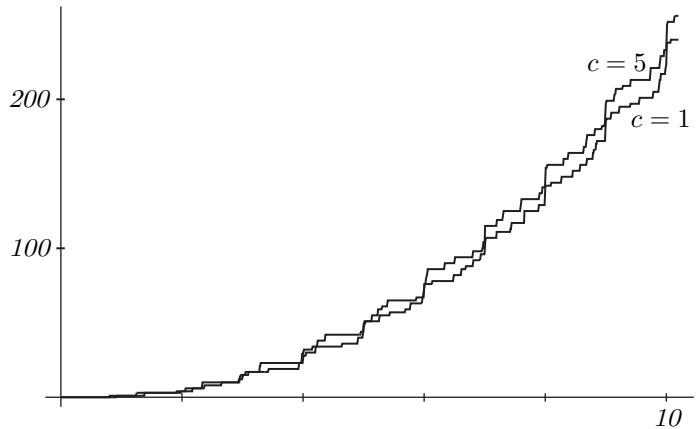} 
\end{center}
\caption{Resonance counting functions for radial step potentials in $\bbH^2$.}\label{Npot}
\end{figure}

Our final example is a ``transparent'' spherical obstacle.  Let $P = \Delta_g$ where 
$$
g = \begin{cases} \kappa^2 g_0 & r < r_0, \\
g_0 & r \ge r_0. \end{cases}.
$$
Then
$$
\calR_P = \bigcup_{k\in \bbZ} \Bigl\{s: \mathcal{W}\bigl[{\bQ^{k}_{s-1}}(r_0), P^{-k}_{\omega(s)}(r_0)\bigr] = 0\Bigr\},
$$
where 
$$
\omega(s) := -\tfrac12 + \sqrt{\kappa^2 s(s-1)+\tfrac14}
$$
Figure~\ref{Ntrans} shows resonance counting functions for $\kappa = \tfrac12$ and $\kappa = \sqrt{2}$.
Once again, Theorem~\ref{main.thm} does not apply because of the lack of smoothness.  However,
in this case the predicted constants, 
$$
B_P = \begin{cases} 2.75 & \kappa = \tfrac12, \\ 
3.70 & \kappa = \sqrt{2}, \end{cases}
$$
at least roughly match the observed behavior, so that one might believe that the theorem would give
a sharp result if extended to this case.
%%fig: Ntrans
\begin{figure} 
\psfrag{10}{$10$}
\psfrag{100}{$100$}
\psfrag{200}{$200$}
\psfrag{300}{$300$}
\psfrag{kh}{$\kappa = \tfrac12$}
\psfrag{k2}{$\kappa = \sqrt{2}$}
\begin{center}  
\includegraphics{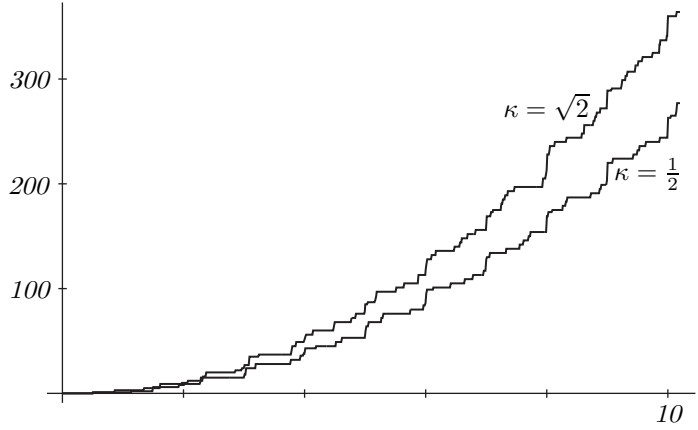} 
\end{center}
\caption{Resonance counting functions for transparent spherical obstacles 
in $\bbH^2$ with $r_0=1$.}\label{Ntrans}
\end{figure}
%

%%%%%%%%
\appendix

%%%%%%%%%
\bigbreak
\section{Legendre function estimates}\label{legend.app}

In this section we will estimate the growth of the Legendre functions
$P_{\nu}^{k}(\cosh r)$ and $\bQ_{\nu}^{k}(\cosh r)$ as $k, |\nu| \to \infty$ simultaneously.
We wish to extract the leading asymptotic behavior, with error bounds
uniform in $\alpha:= (\nu+\tfrac12)/k$ for $\re \alpha \ge 0$.   
The construction of these estimates leans heavily on techniques from Olver \cite[\S11]{Olver}.  

Throughout this discussion we identify $z = \cosh r$ and switch freely between the two
variables.  Let
$$
w(z) = (\sinh r)\> 
\begin{cases}
P_{-\frac12+k\alpha}^{-k}(\cosh r), & \text{ or} \\
\bQ_{-\frac12+k\alpha}^{k}(\cosh r).\\
\end{cases}
$$
Then the Legendre equation reduces to 
\begin{equation}\label{wleg}
\del_z^2 w = (k^2 f + g)\>w,
\end{equation}
with
\begin{equation}\label{fz.gz}
f(r) := \frac{1+\alpha^2\sinh^2 r}{\sinh^4 r},\qquad 
g(r) := -\frac{\sinh^2 r+4}{4 \sinh^4 r}.
\end{equation}
If $\re \alpha = 0$ then the equation (\ref{wleg}) has turning points (points where $f$ vanishes to first order)
when $\alpha = \pm i/\sinh r$.  By conjugation, it suffices to assume $\im \alpha \ge 0$
and so we focus on the upper turning point.
To obtain uniform estimates near this point, we introduce the complex variable $\zeta$ defined by
integrating 
\begin{equation}\label{zeta.fdz}
\sqrt{\zeta}\>d\zeta = \sqrt{f}\>dz,
\end{equation}
starting from $\zeta = 0$ on the left and from $z_0 = \sqrt{1-1/\alpha^2}$ (the turning point)
on the right.  Throughout this section we assume principal branches for the logs and square roots,
under the restriction that $\arg \alpha \in [0,\pi/2]$.

Integrating both sides of (\ref{zeta.fdz}) yields
\begin{equation}\label{zeta.def}
\tfrac23 \zeta^{\frac32} = \phi,
\end{equation}
where
\begin{equation}\label{phi.def}
\begin{split}
\phi(\alpha, r) & :=  \int_{\cosh^{-1} z_0}^r \frac{\sqrt{1+\alpha^2\sinh^2 t}}{\sinh t}\>dt \\
& = 
\alpha \log \left( \frac{\alpha \cosh r +  \sqrt{1 + \alpha^2 \sinh^2 r}}{\sqrt{\alpha^2-1}} \right) 
+ \frac12 \log \left[ 
\frac{\cosh r - \sqrt{1 + \alpha^2 \sinh^2 r}}{\cosh r + \sqrt{1 + \alpha^2 \sinh^2 r}} \right].
\end{split}
\end{equation}
The expression (\ref{phi.def}) is well-defined by principal branches for $\arg \alpha \in (0,\pi/2]$,
and we extend the definition to the positive real axis by continuity.  
(At the apparent singularity at $\alpha = 1$, this extension yields $\phi(1,r)  = \log \sinh r$.)

The region of interest, namely $\arg \alpha \in [0,\pi/2]$ and $r \ge 0$, corresponds to the
sector $\arg \phi \in [-\pi, \tfrac\pi{2}]$, as illustrated in Figure~\ref{phiplot}.   Figure~\ref{zetaplot}
show the corresponding picture for $\zeta$, and illustrates in particular how passing
from $\phi$ to $\zeta$ resolves the singularity at the turning point.  
%%fig: phiplot
\begin{figure}
\psfrag{tz}{$\theta=0$}
\psfrag{tph}{$\theta = \tfrac{\pi}2$}
\psfrag{api}{$r\to \infty$}
\begin{center}  
\includegraphics{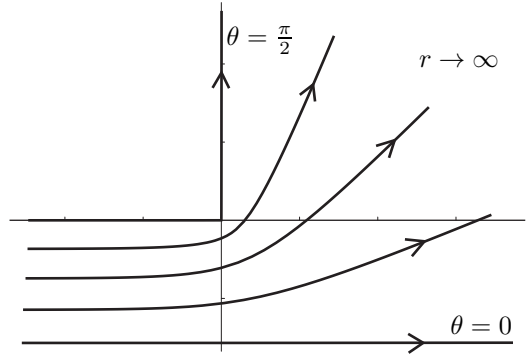} 
\end{center}
\caption{Trajectories of $\phi(\alpha,\cdot)$ with $|\alpha|=2$ and $\theta = \arg \alpha$.
The turning point occurs at the origin.}\label{phiplot}
\end{figure}
%
%%fig: zetaplot
\begin{figure}
\psfrag{tz}{$\theta=0$}
\psfrag{tph}{$\theta = \tfrac{\pi}2$}
\psfrag{api}{$r\to \infty$}
\begin{center}  
\includegraphics{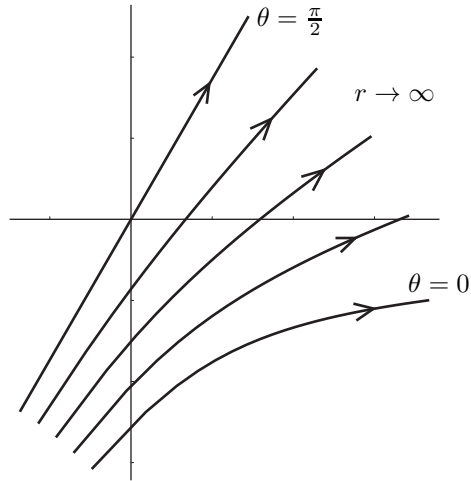} 
\end{center}
\caption{Trajectories of $\zeta(\alpha,\cdot)$ with $|\alpha|=2$ and $\theta = \arg \alpha$.}\label{zetaplot}
\end{figure}
For future reference, we note that $\phi$ satisfies the equation
\begin{equation}\label{dzph}
\del_r \phi = \sqrt{f}\>\sinh r,
\end{equation}
implying in particular that $\re \del_r \phi \ge 0$.  The fact that $\re \phi$ is an increasing function of $r$ will
be important later, and is not so evident from (\ref{phi.def}).

The asymptotics of $\phi(\alpha, \cdot)$ will also play a crucial role.  As $r\to 0$, we have
\begin{equation}\label{phi.asym1}
\phi(\alpha, r) = \log \left(\frac{r}2\right) + p(\alpha) + O(r^2),
\end{equation}
where 
\begin{equation}\label{pal.def}
p(\alpha) := \frac{\alpha}2 \log \left(\frac{\alpha+1}{\alpha-1}\right)
+ \frac12 \log (1-\alpha^2),
\end{equation}
And as $r \to \infty$, we have
\begin{equation}\label{phi.asym2}
\phi(\alpha, r) = \alpha r + q(\alpha) + O(r^{-2}),
\end{equation}
where 
\begin{equation}\label{qal.def}
q(\alpha) := \alpha \log \left(\frac{\alpha}{\sqrt{\alpha^2-1}}\right)
+ \frac12 \log \left(\frac{1-\alpha}{1+\alpha}\right),
\end{equation}

\begin{proposition}\label{legPQ.airy}
Assuming that $k>0$, $\arg \alpha \in [0,\tfrac{\pi}2]$ and $r \in [0, \infty)$, we have 
\begin{equation}\label{p.asym}
P_{-\frac12+k\alpha}^{-k}(\cosh r) =  \frac{2\pi^{\frac12}}{\Gamma(k+1)}
\frac{k^{\frac16} \zeta^{\frac14} e^{\frac{\pi i}6}}{\bigl[1+\alpha^2 \sinh^2 r\bigr]^{\frac14}} 
\>e^{-k p(\alpha)} \Bigl[ \Ai\bigl(k^{\frac23} e^{\frac{2\pi i}3} \zeta\bigr)
+ h_1(k, \alpha, r)\Bigr],
\end{equation}
and
\begin{equation}\label{q.asym}
\bQ_{-\frac12+k\alpha}^{k}(\cosh r) =  
\frac{2\pi}{\Gamma(k\alpha+1)} \frac{k^{\frac16} \zeta^{\frac14} 
(\frac{\alpha}2)^{\frac12}}{\bigl[1+\alpha^2 \sinh^2 r\bigr]^{\frac14}} 
\> e^{kq(\alpha)} \Bigl[ \Ai\bigl(k^{\frac23} \zeta\bigr) + h_0(k,\alpha, r)\Bigr],
\end{equation}
where $\zeta$ is defined by (\ref{zeta.def}) and (\ref{phi.def}), $p(\alpha)$ and $q(\alpha)$ 
are defined in (\ref{pal.def}) and (\ref{qal.def}), respectively.  The error terms satisfy
\begin{equation}\label{hestimates}
\begin{split}
|k^{\frac16} \zeta^{\frac14} h_1(k,\alpha,r)| & \le Ce^{k\re \phi} k^{-1} \Bigl(1+ |\alpha|^{-\frac23}\Bigr),\\
|k^{\frac16} \zeta^{\frac14} h_0(k,\alpha,r)| & \le Ce^{-k\re \phi} k^{-1} \Bigl(1+ |\alpha|^{-\frac23}\Bigr).
\end{split}
\end{equation}
with $C$ independent of both $\alpha$ and $r$.
\end{proposition}
\begin{proof}
If we set $W = (f/\zeta)^{1/4} w$, then the equation (\ref{wleg}) transforms to:
\begin{equation}\label{airy.psi}
\del_\zeta^2 W = (k^2\zeta + \psi)W,
\end{equation}
a perturbed version of the Airy equation, with the extra term given by
\begin{equation}\label{psi.gfz}
\psi =  \frac{\zeta}{4f^2}\del_z^2 f - \frac{5\zeta }{16f^3} (\del_z f)^2 + \frac{\zeta g}{f}
+\frac{5}{16 \zeta^2}.
\end{equation}
Following Olver \cite[Thm.~11.9.1]{Olver}, we consider solutions of the form
\begin{equation}\label{w.airy}
W_\sigma = \Ai(k^{\frac23} e^{\frac{2\pi i \sigma}3} \zeta) + h_\sigma(k, \alpha, r),
\end{equation}
for $\sigma = -1,0,1$, where the error terms satisfy
\begin{equation}\label{hsig.eq}
\del_\zeta^2 h_\sigma - k^2 \zeta h_\sigma = \psi \>\Bigl[h_\sigma + \Ai(k^{\frac23}e^{\frac{2\pi i \sigma}3}\zeta)\Bigr],
\end{equation}

Let us focus first on the Legendre $P$-function.
As $|w| \to \infty$, the Airy function $\Ai(w)$ is exponentially decreasing for $|\arg w| < \tfrac{\pi}3$ 
and exponentially increasing for $|\arg w| \in (\tfrac{\pi}3, \pi]$.  
Since $P_{-\frac12+k\alpha}^{-k}(\cosh r)$ is recessive at zero, and $r \to 0$
corresponds to $\zeta \to e^{-\frac{2\pi i}3}\infty$,
we choose the solution $W_1$ from (\ref{w.airy}).
The assumption that the solution is recessive as $r \to 0$
implies boundary conditions,
\begin{equation}\label{h.bndry1}
h_1|_{r=0} =\del_r h_1|_{r=0} = 0,
\end{equation}
which we must impose on (\ref{hsig.eq}).
To identify the Legendre $P$ function with a multiple of $W_1$, we compare the well-known
asymptotic
$$
P_{-\frac12+k\alpha}^{-k}(\cosh r) = \frac{1}{\Gamma(k+1)} \Bigl(\frac{r}{2} \Bigr)^{k} (1 + O(r^2)),
$$
to the behavior of the ansatz
$$
(\sinh r)^{-1} (\zeta/f)^{\frac14} W_1 = \frac{\zeta^{\frac14}}{[1+\alpha^2 \sinh^2 r]^{\frac14}}
\Bigl(\Ai(k^{\frac23}e^{\frac{2\pi i}3} \zeta) + h_1(k, \alpha, r)\Bigr).
$$

Away from the negative real axis, the Airy function has the asymptotic behavior
\cite[eq.~(4.4.03)]{Olver}
\begin{equation}\label{ai.asym1}
\Ai(w) = \frac{1}{2\pi^{\frac12}} w^{-\frac14} \exp\bigl(-\tfrac23 w^{\frac32}\bigr)\bigl[1 + O(|w|^{-\frac32})\bigr],
\end{equation}
uniformly for $|\arg w| \le \pi-\delta$, with a constant that depends only on $\delta>0$.  
To cover the negative real axis we have also \cite[eq.~(4.4.05)]{Olver},
\begin{equation}\label{ai.asym2}
\Ai(w) = \frac{1}{\pi^{\frac12}} (-w)^{-\frac14} \cos \Bigl(\tfrac23 (-w)^{\frac32}-\tfrac{\pi}4\Bigr) 
\bigl[1 + O(|w|^{-\frac32})\bigr],
\end{equation}
uniformly for $|\arg w| \in [\tfrac{\pi}3 + \delta, \pi]$.  (These estimates agree where they
overlap.)

As $r\to 0$, we have $e^{\frac{2\pi i}3}\zeta \to + \infty$, which is in the range covered by (\ref{ai.asym1}).
Along with the asymptotic behavior of $\zeta$ deduced from (\ref{phi.asym1}), this yields
$$
\zeta^{\frac14}\Ai(k^{\frac23} e^{\frac{2\pi i}3} \zeta) \sim k^{-\frac16} \frac{e^{-\frac{\pi i}6}}{2\pi^{\frac12}} 
e^{kp(\alpha)} \left(\frac{r}2\right)^k \quad\text{as } r \to 0.
$$ 
Thus for the Legendre $P$-function we find
$$
P_{-\frac12+k\alpha}^{-k}(\cosh r) =  2\pi^{\frac12} k^{\frac16} e^{\frac{\pi i}6} \frac{e^{-kp(\alpha)}}{\Gamma(k+1)}
(\sinh r)^{-1} (\zeta/f)^{\frac14} W_1,
$$
which proves (\ref{p.asym}).

To complete the analysis of the $P$ case, it remains to control the size of the error term $h_1(k,\alpha,r)$.  
The error bounds may be derived as in the proof of \cite[Thm.~11.9.1]{Olver},
starting from the differential equation (\ref{hsig.eq}) satisfied by $h_1$.
Using the boundary condition (\ref{h.bndry1}), we can apply variation of parameters to transform
this to an integral equation, 
$$
h_1(k,\alpha, r) = -\frac{2\pi e^{\frac{i\pi}6}}{k^{\frac23}}
\int_0^r K_1(r,r') \psi(r') \>[h_1(k,\alpha,r') + \Ai(k^{\frac23}e^{\frac{2\pi i}3}\zeta(r'))]\>  
\frac{f(r')^\frac12 \sinh r'}{\zeta(r')^\frac12} \>dr',
$$
where
$$
K_1(r,r') := \Ai(k^\frac23e^{\frac{2\pi i}3} \zeta(r'))\Ai(k^\frac23 \zeta(r)) 
- \Ai(k^\frac23 \zeta(r'))\Ai(k^\frac23e^{\frac{2\pi i}3} \zeta(r)).
$$
Then, using the method of successive approximations as in \cite[Thm.~6.10.2]{Olver},  
together with the bounds on the Airy function and its derivatives developed in \cite[\S 11.8]{Olver}, 
we obtain the bound,
\begin{equation}\label{h.bound}
\Bigl| k^{\frac16}\>\zeta^{\frac14} h_1\Bigr| \le Ce^{k \re \phi} \Bigl(e^{ck^{-1} \Psi_1(r)} - 1 \Bigr),
\end{equation}
where 
\begin{equation}\label{Psi.def}
\Psi_1(r) := \int_0^r \left|\psi  f^{\frac12} \zeta^{-\frac12} \right|\> \sinh r'\>dr'.
\end{equation}
Using (\ref{fz.gz}) and (\ref{psi.gfz}), direct computation shows that
\begin{equation}\label{Psi.intg}
\psi f^{\frac12} \zeta^{-\frac12} =
\zeta^{\frac12} \left[ \frac{\alpha^4 \sinh^2 r - 4\alpha^2 \cosh^2 r + 1}{4(1+\alpha^2\sinh^2 r)^{\frac52}} \right]
+\frac{5}{16 } \frac{(1+\alpha^2 \sinh^2 r)^{\frac12}}{\zeta^\frac52\sinh^2 r}.
\end{equation}
Because we require estimates that are uniform in both $\alpha$ and $r$, the analysis of
(\ref{Psi.intg}) is somewhat complicated.  For some small $c>0$,
we will break the estimation into 3 different zones as described below. 
We use the notation $A\asymp B$ to mean that the ratio $A/B$ is bounded
above and below by positive constants that do not depend on $\alpha$ or $r$.

\emph{Zone 1:}
Assume that $|1+\alpha^2 \sinh^2 r| \ge c$ and $|\alpha| \ge 1$.
The first term in the formula (\ref{phi.def}) for $\phi$ dominates for large
$r$ and the second term for small $r$.  We can thus derive the bounds, 
$$
|\phi| \asymp \begin{cases} - \log |\alpha|r &\text{for } |\alpha| \sinh r \le \tfrac12 \\
|\alpha|r &\text{for } |\alpha| \sinh r \ge \tfrac12, \end{cases}
$$
Using this to estimate $\zeta = (\tfrac32 \phi)^{\frac23}$ in (\ref{Psi.intg}) gives
$$
\left|\psi  f^{\frac12} \zeta^{-\frac12} \right|\> \sinh r
\le \begin{cases}
C_1 |\alpha|^2 (-\log |\alpha| r)^{\frac13} r + C_2 (-\log |\alpha| r)^{-\frac53} r^{-1} & \text{for } |\alpha| \sinh r \le \tfrac12, \\
C_1 |\alpha|^{-\frac23} r^{\frac13} e^{-2r} + C_2 |\alpha|^{-\frac23} r^{-\frac53}& \text{for } |\alpha| \sinh r \ge \tfrac12.
\end{cases}
$$
It is then relatively straightforward to control the contribution of these terms to (\ref{Psi.def}).  
For $|\alpha| \ge 1$, we obtain
\begin{equation}\label{zone1}
\int_{|1+\alpha^2 \sinh^2 r| \ge c} \left|\psi  f^{\frac12} \zeta^{-\frac12} \right|\> \sinh r\>dr \le C,
\end{equation}
for some $C$ that depends on $c$ but not on $\alpha$.

\emph{Zone 2:}
Assume that $|1+\alpha^2 \sinh^2 r| \ge c$ and $|\alpha| \le 1$.
In this case we claim that
$$
|\phi| \asymp \begin{cases} |\log(1 - e^{-r})| &\text{for } |\alpha| \sinh r \le \tfrac12 \\
|\alpha|\bigl(r+ \log 2|\alpha|\bigr) &\text{for } |\alpha| \sinh r \ge \tfrac12, \end{cases}
$$
Using these in conjunction with (\ref{Psi.intg}) then gives
$$
\left|\psi  f^{\frac12} \zeta^{-\frac12} \right|\> \sinh r
\le \begin{cases}
C_1 |\log(1 - e^{-r})|^{\frac13} \sinh r + C_2 |\log(1 - e^{-r})|^{-\frac53} (\sinh r)^{-1} & \text{for } |\alpha| \sinh r \le \tfrac12, \\
C_1 |\alpha|^{-\frac23} (r+ \log 2|\alpha|)^{\frac13} e^{-2r} + C_2 |\alpha|^{-\frac23} 
\bigl(r+ \log 2|\alpha|\bigr)^{-\frac53}& \text{for } |\alpha| \sinh r \ge \tfrac12.
\end{cases}
$$
For $|\alpha| \le 1$, we obtain
\begin{equation}\label{zone2}
\int_{|1+\alpha^2 \sinh^2 r| \ge c} \left|\psi  f^{\frac12} \zeta^{-\frac12} \right|\> \sinh r\>dr \le C|\alpha|^{-\frac23},
\end{equation}
for some $C$ that depends on $c$ but not on $\alpha$.

\emph{Zone 3:}
Assume that $|1+\alpha^2 \sinh^2 r| \le c$.  This puts us near the turning point.
It is convenient to use the $z = \cosh r$ variable here.  The turning point occurs 
at the point
$$
z_0 := \sqrt{1-\alpha^{-2}},
$$
which lies near the path of integration for (\ref{Psi.def}) only when $\arg\alpha$ is close to $\tfrac{\pi}2$.  
Note that
$$
1+\alpha^2 \sinh^2 r = \alpha^2 (z^2-z_0^2),
$$
so that the assumption $|1+\alpha^2 \sinh^2 r| \le c$ translates to 
$$
|z-z_0| = \begin{cases}O(|\alpha|^{-1}) & \text{for }|\alpha| \le 1\\
O(|\alpha|^{-2}) & \text{for }|\alpha| \ge 1 \end{cases}
$$

To obtain estimates near the turning point, we introduce the functions
$$
p(z) := \left( \frac{f}{z-z_0}\right)^\frac12, \quad q(z) =  \frac{\phi}{(z-z_0)^\frac32}.
$$
By rewriting $p(z)$ in the form
$$
p(z) = \frac{\alpha\sqrt{z + z_0}}{z^2 -1},
$$
we can easily obtain estimates, 
\begin{equation}\label{pk.est}
|p^{(k)}(z)| \asymp 
\begin{cases} |\alpha|^{\frac52+k} & \text{for }|\alpha| \le 1\text{ and }|z-z_0| = O(|\alpha|^{-1}),\\
|\alpha|^{3+2k} & \text{for }|\alpha| \ge 1\text{ and }|z-z_0| = O(|\alpha|^{-2}). \end{cases}
\end{equation}
Using the definition of $\phi$ as $\int_{z_0}^z \sqrt{f}\>dz$, we can write $q(z)$ in the form
$$
q(z) = \int_{0}^1 t^\frac12 p(z_0 + t(z- z_0)) \>dt.
$$
Then from (\ref{pk.est}) we can derive estimates of the same form for $q(z)$,
\begin{equation}\label{qk.est}
|q^{(k)}(z)| \asymp 
\begin{cases} |\alpha|^{\frac52+k} & \text{for }|\alpha| \le 1\text{ and }|z-z_0| = O(|\alpha|^{-1}),\\
|\alpha|^{3+2k} & \text{for }|\alpha| \ge 1\text{ and }|z-z_0| = O(|\alpha|^{-2}). \end{cases}
\end{equation}

Using (\ref{pk.est}) and (\ref{qk.est}), with the fact that
$f/\zeta = p^2 (\tfrac32 q)^{-\frac23}$ and the formula for $\psi$ given in (\ref{psi.gfz}),
we obtain
$$
\left|\psi  f^{\frac12} \zeta^{-\frac12} \right| \le 
\begin{cases} O(|\alpha|^{\frac13}) & \text{for }|\alpha| \le 1\text{ and }|z-z_0| = O(|\alpha|^{-1}),\\
O(|\alpha|^{2}) & \text{for }|\alpha| \ge 1\text{ and }|z-z_0| = O(|\alpha|^{-2}). \end{cases}
$$
For $|\alpha| \le 1$, the result is 
\begin{equation}\label{zone3a}
\int_{|z-z_0| \le C|\alpha|^{-1}} \left|\psi  f^{\frac12} \zeta^{-\frac12} \right|\>dz = O(|\alpha|^{-\frac23}),
\end{equation}
For $|\alpha|\ge 1$ the corresponding estimate is
\begin{equation}\label{zone3b}
\int_{|z-z_0| \le C|\alpha|^{-2}} \left|\psi  f^{\frac12} \zeta^{-\frac12} \right|\>dz = O(1).
\end{equation}

Now we can combine the estimates of contributions to $\Psi(r)$ from all three zones, 
namely (\ref{zone1}), (\ref{zone2}), (\ref{zone3a}), and (\ref{zone3b}), to obtain
\begin{equation}\label{Psi.est}
|\Psi_1(r)| \le C(1 + |\alpha|^{-\frac23}),
\end{equation}
for $\arg \alpha \in [0, \tfrac{\pi}2]$ and $r \in [0,\infty)$, with $C$ independent of both $r$ and $\alpha$.
Applying the resulting estimate of $\Psi(r)$ in (\ref{h.bound}) then gives
$$
\Bigl| k^{\frac16}\>\zeta^{\frac14} h_1\Bigr| \le Ck^{-1} \Bigl(1+ |\alpha|^{-\frac23}\Bigr) e^{k \re \phi}.
$$
This completes the error analysis in the $P$ case.

We turn now to the Legendre $Q$-function and the proof of (\ref{q.asym}).  
We want the solution to be recessive at $r = \infty$,
so we set $\sigma = 0$ in the ansatz (\ref{w.airy}) and impose
the condition
\begin{equation}\label{h2.bc}
h_0 = O(r^{-2}), \quad\text{ as }r \to \infty.
\end{equation}
Using (\ref{phi.asym2}) and (\ref{ai.asym1}) we have 
$$
\frac{\zeta^{\frac14}}{[1 + \alpha^2 \sinh^2 r]^{\frac14}} \Ai(k^{\frac23}\zeta) \sim
\frac{k^{-\frac16}(\frac{\alpha}2)^{-\frac12}}{2\pi^{\frac12}} e^{-(k\alpha+\frac12)r},
$$
as $r \to \infty$.  
From the asymptotic (\ref{Q.asym}) we find
$$
\bQ_{-\frac12+k\alpha}^{k}(\cosh r) \sim \frac{\pi^{\frac12}}{\Gamma(k\alpha +1)} e^{-(k\alpha+\frac12)r}.
$$
Hence, for the $Q$-Legendre function we have
$$
\bQ_{-\frac12+k\alpha}^{k}(\cosh r) = \frac{2\pi k^{\frac16} (\frac{\alpha}2)^{\frac12}}{\Gamma(k\alpha +1)}
(\sinh r)^{-1} (\zeta/f)^{\frac14} W_0,
$$
which proves (\ref{q.asym}).

To control $h_0$ we use the boundary condition (\ref{h2.bc}) to transform the differential equation (\ref{hsig.eq})
for $h_0$ into an integral equation,
$$
h_0(k,\alpha, r) = \frac{2\pi e^{-\frac{i\pi}6}}{k^{\frac23}}
\int_r^\infty K_0(r,r') \psi(r') \>[h_0(k,\alpha,r') + \Ai(k^{\frac23}\zeta(r'))]\>  
\frac{f(r')^\frac12 \sinh r'}{\zeta(r')^\frac12} \>dr',
$$
where
$$
K_0(r,r') := \Ai(k^\frac23 \zeta(r'))\Ai(k^\frac23 e^{-\frac{2\pi i}3}\zeta(r)) 
- \Ai(k^\frac23 e^{-\frac{2\pi i}3}\zeta(r'))\Ai(k^\frac23 \zeta(r)).
$$

The consequence is that the analog of (\ref{h.bound}) for $h_2$ is
\begin{equation}\label{h2.bound}
\Bigl| k^{\frac16}\>\zeta^{\frac14} h_0\Bigr| \le Ce^{-k \re \phi} \Bigl(e^{ck^{-1} \Psi_0(r)} - 1 \Bigr),
\end{equation}
with 
$$
\Psi_0(r) := \int_r^\infty \left|\psi  f^{\frac12} \zeta^{-\frac12} \right|\> \sinh r'\>dr'.
$$
Since $\Psi_0(r) = \Psi_1(\infty) - \Psi_1(r)$, we can simply apply the estimate (\ref{Psi.est}) from the $P$ case to
complete the proof.
\end{proof}

\bigbreak
The first application we need from Proposition~\ref{legPQ.airy} is a set of good upper
bounds.  

\begin{corollary}\label{PQ.bounds}
Assuming that $|k\alpha| \ge 1$, $\re \alpha \ge 0$, and $r \in [r_0,r_1]$, we have the following estimates: 
\begin{equation}\label{p.upper}
\Bigl| P_{-\frac12+k\alpha}^{-k}(\cosh r) \Bigr| \le  \frac{Ck^{\frac16}}{\Gamma(k+1)} 
\> e^{k \re [\phi(\alpha,r) - p(\alpha)]}
\end{equation}
and
\begin{equation}\label{q.upper}
\Bigl| \bQ_{-\frac12+k\alpha}^{k}(\cosh r) \Bigr| \le  
\frac{Ck^{\frac16} |\alpha|^{\frac12}}{|\Gamma(k\alpha+1)|} \> e^{-k\re[\phi(\alpha, r) - q(\alpha)]},
\end{equation}
where $C$ depends only on $r_0$ and $r_1$.
\end{corollary}
\begin{proof}
By conjugation, it suffices to assume that $\arg \alpha \in [0,\tfrac{\pi}2]$.
Using the asymptotics, (\ref{ai.asym1}) and (\ref{ai.asym2}), and the first error estimate from (\ref{hestimates}),
we have
\begin{equation}\label{Aih.est}
\Bigl|k^{\frac16} \zeta^{\frac14}  \bigl[ \Ai\bigl(k^{\frac23} e^{\frac{2\pi i}3} \zeta\bigr)
+ h_1(k, \alpha, r)\bigr] \Bigr|
\le C e^{k\re \phi}
\end{equation}
If we assume that $|1 + \alpha^2 \sinh^2 r| \ge c$, for some $c>0$, 
then the estimate (\ref{p.upper}) follows immediately from (\ref{p.asym}).

On the other hand, if $|1 + \alpha^2 \sinh^2 r| \le c$, then by the assumption that $r \in [r_0, r_1]$,
we deduce that $|\alpha|$, $|\phi|$, and the ratio $\zeta/[1+\alpha^2 \sinh^2 r]$ are all $O(1)$.
For $|\phi| \ge k^{-1}$, we can use (\ref{Aih.est}) to complete the estimate.  If $|\phi| < k^{-1}$,
then $|\Ai(k^{\frac23} e^{\frac{2\pi i}3} \zeta) + h_1(k, \alpha, r)|$ is bounded by (\ref{hestimates})
and the fact that $\Ai(w)$ is regular at the origin.  (It is only because of this last case that the 
factor $k^{\frac16}$ must be included in the final estimate.)

The argument for (\ref{q.upper}) is essentially identical.
\end{proof}

\bigbreak
Our second application of Proposition~\ref{legPQ.airy} is to control the ratio of Legendre functions.
\begin{corollary}\label{PoverQ}
Assuming that $|k\alpha| \ge 1$, $\vep>0$, and $\arg \alpha \in [0,\tfrac{\pi}2-\vep]$, we have
uniform bounds for $k$ sufficiently large:
$$
\left| \frac{P_{-\frac12+k\alpha}^{-k}(\cosh r)}{\bQ_{-\frac12+k\alpha}^{k}(\cosh r)} \right|
\asymp \left| \frac{\Gamma(k\alpha+1)}{\alpha^{\frac12}\Gamma(k+1)} \right| 
e^{k \re[2\phi(\alpha,r) - p(\alpha) - q(\alpha)]},
$$
meaning the the ratio of the two sides is bounded above and below by constants depending only on $\vep$.
(The upper bound extends to $\arg \alpha = \tfrac{\pi}2$, but the lower bound does not.)
\end{corollary}
\begin{proof}
By  (\ref{p.asym}) and (\ref{q.asym}) we have
\begin{equation}\label{PQratio}
\frac{P_{-\frac12+k\alpha}^{-k}(\cosh r)}{\bQ_{-\frac12+k\alpha}^{k}(\cosh r)} =
c\>\frac{\alpha^{\frac12} \Gamma(k\alpha+1)}{\Gamma(k+1)}
\> e^{-k[p(\alpha) + q(\alpha)]}
\>\frac{\Ai\bigl(k^{\frac23} e^{\frac{2\pi i}3} \zeta\bigr) + h_1(k, \alpha, r)}
{\Ai\bigl(k^{\frac23} \zeta\bigr) + h_2(k,\alpha, r)}
\end{equation}
For some $c>0$, consider first the case where $|k \phi| > c$.  
The assumption that $\arg \alpha$ is bounded away from $\tfrac{\pi}2$
implies that $\arg \zeta \in [-\tfrac{2\pi}3, \tfrac{\pi}3 - \vep_1]$, so that we can apply (\ref{ai.asym1})
to estimate both of the Airy functions in  (\ref{PQratio}).  By choosing $c$
sufficiently large, we can assume that the factor $1+ O(|w|^{-3/2})$ appearing in (\ref{ai.asym1}) is bounded away
from zero, since $\tfrac23 |w|^{\frac32} = |k\phi|$ in our case.  By the estimates (\ref{hestimates}) and
the assumption $|k\alpha| \ge 1$, by choosing $k$ sufficiently large we can assume that $|h_1|$ and $|h_2|$
are arbitrarily small relative to the Airy function estimates.  Under these assumptions we have
$$
\left| \frac{\Ai\bigl(k^{\frac23} e^{\frac{2\pi i}3} \zeta\bigr) + h_1(k, \alpha, r)}
{\Ai\bigl(k^{\frac23} \zeta\bigr) + h_2(k,\alpha, r)} \right| \asymp e^{2k\re \phi(\alpha,r)}.
$$
The bound then follows immediately.

If $|k \phi| \le c$, then because $\Ai(w)$ is non-zero near the origin, the bound follows immediately 
from (\ref{p.asym}), (\ref{q.asym}), and (\ref{hestimates}), provided that
$k$ is sufficiently large.
\end{proof}

%\bibliographystyle{amspl-db.bst}
%\bibliography{refb.bib}

\end{document}